\documentclass[12pt]{article}

\usepackage[utf8]{inputenc}
\usepackage{amsmath, mathtools}
\usepackage{amsfonts, bbm, mathrsfs, yfonts}
\usepackage{amssymb, amsthm, upgreek, bbm}
\usepackage[ruled, linesnumbered, nosemicolon]{algorithm2e}
\usepackage{enumitem, hyperref,xcolor, textcmds}
\usepackage{subcaption}
\usepackage[sort]{natbib}
\usepackage{todonotes}
\allowdisplaybreaks

\theoremstyle{plain}
\newtheorem{theorem}{Theorem}
\newtheorem{lemma}[theorem]{Lemma}
\newtheorem{corollary}[theorem]{Corollary}

\theoremstyle{definition}
\newtheorem{assumption}{Assumption}

\theoremstyle{remark}
\newtheorem{remark}[theorem]{Remark}
\newtheorem{example}[theorem]{Example}

\renewcommand{\d}{\mathrm{d}}
\renewcommand{\complement}{\texttt{C}}

\title{Uniform ergodicity of geodesic slice sampling} 
\author{Mareike Hasenpflug\thanks{University of Passau, Germany, Email: mareike.hasenpflug@uni-passau.de}}

\begin{document}

\maketitle

\begin{abstract}
	Geodesic slice sampling, introduced in \citep{durmus2023geodesic}, is a slice sampling based Markov chain Monte Carlo method for approximate sampling from distributions on Riemannian manifolds.
	We prove that it is uniformly ergodic for distributions with compact support that have a bounded density with respect to the Riemannian measure.
	The constants in our convergence bound are available explicitly, and we investigate their dependence on the hyperparameters of the geodesic slice sampler, the target distribution and the underlying domain.
\end{abstract}

\section{Introduction}
Sampling tasks, appearing for example in Bayesian inference, numerical integration or optimization, can be performed with Markov chain Monte Carlo methods.
In contrast to tailor made samplers for specific distributions, 
the strategy of these methods is to approximate the distribution of interest with a Markov chain,
which achieves a wide applicability.
At the same time this renders it important to understand how fast the approximation of the target distribution takes place in order to assess the performance.
In the language of Markov chains, the convergence of the $n$-step distributions, i.e. the distribution of the $n$th state of the Markov chain,
is termed ergodicity.
Crucially this depends on the Markov kernel of the chain, which essentially describes the transition mechanism used by the Markov chain to pass from one state to the next.
Different notions of ergodicity have established to qualitatively and quantitatively describe 
this convergence.
They range from pure ergodicity, 
i.e.\ convergence of the $n$-step distributions to the target distribution in total variation distance for all starting points,
but with an unknown and therefore potentially arbitrarily slow rate,
to comparably strong properties like uniform ergodicity,
which provides an exponential convergence speed and worst case convergence guarantees at the same time.
Expressed precisely, a Markov kernel $K$ on a measurable space $(\mathsf{X}, \mathcal{X})$ is uniformly ergodic if there exist constants $\upvarrho \in [0, 1)$ and $C \in [0, \infty)$ such that
\begin{equation}\label{Eq: genreal uniform ergodicity}
	\sup_{x \in \mathsf{X}} d_{\mathrm{tv}}\big(K^n(x, \cdot), \pi\big) \leq C \upvarrho^n \qquad \text{for all } n \in \mathbb{N},
\end{equation}
where $d_{\mathrm{tv}}$ denotes the total variation distance and $\pi$ is the invariant distribution of $K$.
Beyond the strong statement about the decrease of the total variation distance between the $n$-step distributions of the Markov chain and the target distribution,
uniform ergodicity implies many other desirable properties, e.g.\
asymptotic and explicit error bounds on the mean squared integration error  \citep{mathe1999numerical} and \citep[Theorem 3.34]{rudolf2012explicit},
a Hoeffding type inequality \citep{glynn2002hoeffding},
or a central limit theorem \citep{jones2004on}.
For several Markov chain Monte Carlo methods, conditions under which they are uniformly ergodic are known.
For example an independent Metropolis Hastings sampler is uniformly ergodic if and only if the ratio of proposal and target density is bounded away from zero \citep{mengersen1996rates},
and an ideal slice sampler is uniformly ergodic if the target density is bounded and the reference measure is finite \citep{mira2002efficiency}.

In this paper, we establish conditions for the uniform ergodicity of geodesic slice sampling.
Geodesic slice sampling, introduced in \citep{durmus2023geodesic}, 
can be used to approximately sample from distributions $\pi$ on a Riemannian manifold that are absolutely continuous with respect to the Riemannian measure.
The transition mechanism of this gradient free method works by choosing a random unit speed geodesic through the current point and then running a 1-dimensional hybrid slice sampler following the ideas of  \cite{neal2003slice} along the random geodesic.
Thus, it can be considered as a hybridisation of the ideal slice sampler on Riemannian manifolds that uses the Riemannian measure as reference measure \citep{latuszynski2024convergence}.

\cite{durmus2023geodesic} show that under mild regularity assumptions the geodesic slice sampler is reversible with respect to the target distribution $\pi$ and its $n$-step distributions converge to it in total variation distance.
However, nothing is known about the speed of this convergence so far,
calling for more quantitative convergence statements as the manifold setting becomes increasingly popular in sampling \citep{diaconis2013sampling, brubaker2012family, zappa2018monte, yang2024stereographic}.
Although among notions of (quantitative) convergence uniform ergodicity is a strong statement, 
other (weaker) forms of convergence may be often more useful in practice due to the equally strong assumptions usually required for uniform ergodicity to hold (see e.g. the conditions for independent Metropolis Hastings algorithm and ideal slice sampling above).
Also for geodesic slice sampling we require a suitable compact setting in order to establish uniform ergodicity.
However, when working with manifolds, this setting arises much more natural than in Euclidean space.
For instance the Stiefel manifold, the Grassmann manifold or the Euclidean unit sphere are compact Riemannian manifolds which naturally appear in several applications \citep{mantoux2021understanding, mardia09directional, marignier2023posterior, lui2012advances, tsilifis2018bayesian, birdal2018bayesian}.
This inherent wider applicability renders uniform ergodicity an even more desirable property in a Riemannian manifold setting.

Describing the content of this paper in more detail, we show that for distributions with compact support and bounded density the geodesic slice sampler is uniformly ergodic and provide explicit constants such that \eqref{Eq: genreal uniform ergodicity} holds.
In particular, we get an insight how our convergence guarantee depends on the choice of the hyperparameters of the geodesic slice sampler and how they should be chosen to get the best bound.
Moreover, we see that the main influence of the target distribution onto the guaranteed convergence rate $\upvarrho$ is through its level set function,
a quantity that is known to control the convergence behaviour of ideal slice samplers (see Remark \ref{R: Dependence of ergodicity constant on target distirbution}).
Note that we also extend the set of allowable hyperparameters for the geodesic slice sampler.

We set our uniform ergodicity result into context with what is known about the convergence behaviour of other Markov chain Monte Carlo methods for distributions on Riemannian manifolds.
\cite{takeda2023geometric} show that for compact manifolds the  Hamiltonian Monte Carlo method for distributions on embedded submanifolds of $\mathbb{R}^d$ that are the preimage of a point under a smooth defining map proposed by \cite{brubaker2012family} is uniformly ergodic.
Notions of convergence differing from uniform ergodicity are treated by the subsequent works.
Beyond the compact case, \cite{brubaker2012family} provide regularity conditions under which their Hamiltonian Monte Carlo method is still ergodic.
By \cite{whalley2024randomized}, this result is extended to a Hamiltonian Monte Carlo method that uses a random integration duration.
For a geodesic-based Metropolis-Hastings algorithm, \cite{goyal2019sampling} show mixing time bounds from a warm start for uniform distributions on geodesically convex sets on manifolds with non-negative sectional curvature.
Moreover, \cite{cheng2022efficient} establish Wasserstein-1-bounds in a bounded curvature setting for a Markov chain Monte Carlo method that is based on the discretisation of a Langevin diffusion.

The rest of this paper is organised a follows.
We start our investigation with some notation and a brief introduction to the Riemannian setting in the next section.
Our main result, including proof and discussion, is contained in Section \ref{S: main result}.
There the reader can also find a brief recapitulation of the geodesic slice sampler.
The appendix provides some auxiliary results about the stepping-out and shrinkage procedure, which is used within the geodesic slice sampler.

\section{Notation and manifold setting}
\paragraph{Basic notation.}
In the following two paragraphs we cover some general notation used throughout this paper.
We use $\mathbb{N}$ to denote the strictly positive integers and $d \in \mathbb{N}$ to denote the dimension of a space, when applicable.
The $(d-1)$-dimensional Euclidean unit sphere in $\mathbb{R}^d$ is denoted as $\mathbb{S}^{d-1}$. 
For a subset $\mathsf{A} \subseteq \mathbb{R}^d$ we write $\mathrm{conv}(\mathsf{A})$ for its convex hull,
and for the supremum norm of a function $f: \mathsf{X} \to \mathbb{R}$ we write $\|f\|_\infty:= \sup_{x \in \mathsf{X}} |f(x)|$.
Moreover, let $\partial_i\vert_{z}$ be the partial derivative with respect to the $i$th coordinate evaluated at $z \in \mathbb{R}^{d}$.
Given a topological space $\mathsf{X}$, the closure of a set $\mathsf{A} \subseteq \mathsf{X}$ is denoted as $\overline{\mathsf{A}}$ and its interior as $\mathrm{int}(\mathsf{A})$.
In the following, we always assume that $\mathbb{R}^d$ is equipped with the standard topology,
and subsets of $\mathbb{R}^d$ inherit the corresponding subspace topology.
If we operate on a metric space $(\mathsf{X}, m)$, we set $\mathrm{diam}(\mathsf{A}) := \sup_{x,y \in \mathsf{A}} m(x,y) \in [0, \infty]$ for non-empty subsets $\mathsf{A} \subseteq \mathsf{X}$.
Moreover, define for $x \in \mathsf{X}$ and $r > 0$ the metric ball $B_r(x):= \{ y \in \mathsf{X} \mid m(x,y) < r\}$.
In the following, we always assume that $\mathbb{R}^d$ is equipped with the standard inner product.

We turn to some general probabilistic notation.
For a given topological space $\mathsf{X}$, we denote by $\mathcal{B}(\mathsf{X})$ its Borel-$\sigma$-algebra.
Moreover, the $d$-dimensional Lebesgue measure is denoted as $\mathrm{Leb}_{d}$.
For a set $\mathsf{A} \in \mathcal{B}(\mathbb{R}^d)$ with $\mathrm{Leb}_{d}(\mathsf{A})\in (0, \infty)$, let $\mathrm{Unif}(\mathsf{A}) := \mathrm{Leb}_{d}(\mathsf{A})^{-1} \mathrm{Leb}_{d}(\cdot \cap \mathsf{A})$ be the uniform distribution on $\mathsf{A}$.
Similarly, if $\mathsf{A}$ has finitely many elements, we set $\mathrm{Unif}(\mathsf{A})$ to be the discrete uniform distribution on $\mathsf{A}$.
Let $(\mathsf{X}, \mathcal{X})$ and $(\mathsf{Y}, \mathcal{Y})$ be two measurable spaces.
For a measure $\mu$ on $(\mathsf{X}, \mathcal{X})$ and a measurable function $f: \mathsf{X} \to \mathsf{Y}$,
we denote by $f_\sharp(\mu) := \mu(f^{-1}(\cdot))$ the pushforward measure on $(\mathsf{Y}, \mathcal{Y})$ of $\mu$ under $f$.
Any random variables that appear in the following are assumed to be defined on some rich enough probability space $(\Omega, \mathcal{A}, \mathbb{P})$.
Given a probability measure $\mu$, we use $Z \sim \mu$ to express that the random variable $Z$ is distributed according to $\mu$.
If $\nu$ is a second probability measure on the same measurable space $(\mathsf{X}, \mathcal{X})$ as $\mu$,
then we denote by $d_{\mathrm{tv}}(\mu, \nu):= \sup_{\mathsf{A} \in \mathcal{X}} |\mu(\mathsf{A})- \nu(\mathsf{A})|$ the total variation distance between $\mu$ and $\nu$.
Finally, we say that a mapping $K: \mathsf{X} \times \mathcal{X} \to [0,1]$ is a transition or Markov kernel if $K(x, \cdot)$ is a probability measure on $(\mathsf{X}, \mathcal{X})$ for all $x \in \mathsf{X}$ and the function $K(\cdot, \mathsf{A}): \mathsf{X} \to [0,1]$ is measurable for all $\mathsf{A} \in \mathcal{X}$.
Then we inductively define the transition kernels $K^1 := K$ and $K^{n+1}(x, \mathsf{A}) := \int_{\mathsf{X}} K^n(y, \mathsf{A}) \, K(x, \d y)$, $x \in \mathsf{X}, \mathsf{A} \in \mathcal{X}$, for all $n \in \mathbb{N}$.
 
\paragraph{Manifold setting.}
The rest of this section outlines the notions from differential geometry relevant for this paper and fixes the used notation.
We purposely remain brief here and refer the reader to \citep{boothby1986introduction, lee2013introduction, lee2018introduction} 
for an introduction to differential geometry or to \citep[Section 2.1]{hasenpflug2024slice} for a more detailed overview of the content of this section.

A $d$-dimensional manifold is a second countable Hausdorff space such that for every point there exists an open neighbourhood that is homeomorphic to an open subset of $\mathbb{R}^d$.
In this paper we consider Riemannian manifolds, which are smooth manifolds that are equipped with a smooth field of positive definite, symmetric bilinear forms. This smooth field of bilinear forms, called the Riemannian metric, induces a natural measure on the manifold, called the Riemannian measure.
For a $d$-dimensional Riemannian manifold $\mathsf{M}$, which is in particular a topological space, 
it is defined on the Borel-$\sigma$-algebra $\mathcal{B}(\mathsf{M})$.
For all points $x \in \mathsf{M}$, we denote by $T_x\mathsf{M}$ the tangent space to $\mathsf{M}$ at $x$.
To characterise the Riemannian measure on a coordinate neighbourhood $(\mathsf{U}, \varphi)$ of $\mathsf{M}$,
we need the coordinate frames $E_1^{\varphi}, \ldots, E_d^{\varphi}$ induced by $(\mathsf{U}, \varphi)$, which are $d$ vector fields on $\mathsf{U}$ given by
\[
E_i^{\varphi}: \mathsf{U} \to \bigcup_{x \in \mathsf{M}} T_x\mathsf{M}, \qquad x \mapsto E_{i,x}^{\varphi}:= \varphi^{-1}_\ast\left (\partial_i\vert_{\varphi(x)}\right ) \in T_x\mathsf{M}, \qquad i \in \{1, \ldots, d\},
\]
where $F_\ast$ denotes the differential of a $C^\infty$-map $F$ between two smooth manifolds. 
Observe that $E_{1,x}^{\varphi}, \ldots, E_{d,x}^{\varphi}$ is a basis of the vector space $T_x\mathsf{M}$ for all $x \in \mathsf{U}$,
such that their Gram matrix $\left(\mathfrak{g}_x(E_{i,x}^{\varphi}, E_{j, x}^{\varphi})\right)_{\{1\leq i,j \leq d\}}$ in the inner product space $(T_x\mathsf{M}, \mathfrak{g}_x)$ is regular.
We define
\[
	\sqrt{\det(\mathfrak{g},\varphi)}: \mathsf{U} \to (0, \infty), \qquad x \mapsto \sqrt{\det\left[\left(\mathfrak{g}_x(E_{i,x}^{\varphi}, E_{j, x}^{\varphi})\right)_{\{1\leq i,j \leq d\}}\right]},
\]
where $\det$ denotes the determinant of a matrix.
Roughly speaking, this function captures the infinitesimal deformation when passing from $\mathbb{R}^d$ to $\mathsf{M}$ via the homeomorphism $\varphi^{-1}$.
Due to the second countability of $\mathsf{M}$, 
there exists a countable collection of coordinate neighbourhoods $\{(\mathsf{U}_i, \varphi_i)\}_{i \in \mathbb{N}}$ such that $\mathsf{M} = \bigcup_{i \in \mathsf{N}} \mathsf{U}_i$, which we call a countable atlas.
Let $\{\rho_i: \mathsf{M} \to [0, \infty)\}_{i \in \mathbb{N}}$ be a partition of unity subordinate to $\{\mathsf{U}_i\}_{i \in \mathbb{N}}$,
i.e.\ a partition of unity such that the support of $\rho_i$ is a subset of $\mathsf{U}_i$ for all $i \in \mathbb{N}$.
The Riemannian measure is then given by
\[
	\nu_{\mathfrak{g}}(\mathsf{A}) = \sum_{i \in \mathbb{N}} \int_{\varphi_i(\mathsf{U}_i)} \left(\rho_i \cdot \mathbbm{1}_{\mathsf{A}} \cdot \sqrt{\det(\mathfrak{g}, \varphi_i)}\right) \circ \varphi_i^{-1}(z) \ \mathrm{Leb}_{d}(\d z), \qquad \mathsf{A} \in \mathcal{B}(\mathsf{M}).
\]
This construction is independent of the choice of the countable atlas and the partition of unity.

In addition to the Riemannian measure, the Riemannian metric $\mathfrak{g}$ induces also a metric $\mathrm{dist}$ on $\mathsf{M}$.
The topology from this metric agrees with the topology on $\mathsf{M}$, turning $\mathsf{M}$ into a metric space.
Consequently also on a Riemannian manifold metric balls and the diameter of a set are defined.

We cover some more objects from Riemannian geometry needed throughout this paper.
\begin{enumerate}
	\item \textbf{Unit tangent bundle:}
		The tangent bundle $T\mathsf{M}:= \bigcup_{x \in \mathsf{M}} \{x\} \times T_x\mathsf{M}$ equipped with the Sasaki metric $\mathfrak{G}$ is a $2d$-dimensional Riemannian manifold.
		For $x \in \mathsf{M}$ we denote by 
		\[
		U_x\mathsf{M}:= \{ v \in T_x\mathsf{M} \mid \mathfrak{g}_x(v,v) = 1\}
		\]
		the unit sphere in the inner product space $(T_x\mathsf{M}, \mathfrak{g}_x)$, called the unit tangent spheres.
		Their (disjoint) union forms an embedded submanifold of $T\mathsf{M}$ called the unit tangent bundle $U\mathsf{M}:= \bigcup_{x \in \mathsf{M}} \{x\} \times U_x\mathsf{M}$.
		Moreover, also the unit tangent spheres inherit a Riemannian structure from $T\mathsf{M}$ as embedded submanifolds. 
		Denoting the resulting Riemannian metric on $U_x\mathsf{M}$ as $\iota^\ast\mathfrak{g}^{(x)}$ for $x \in \mathsf{M}$,
		this yields a Riemannian measure $\nu_{\iota^\ast\mathfrak{g}^{(x)}}$ on $U_x\mathsf{M}$.
		Equip the Euclidean unit sphere $\mathbb{S}^{d-1}$ with the standard or round metric $\widehat{\mathfrak{g}}$,
		and denote by $\upomega_{d-1}:= \nu_{\widehat{\mathfrak{g}}}(\mathbb{S}^{d-1})$ the volume of the Euclidean unit sphere.
		We have the following relation of $\nu_{\iota^\ast\mathfrak{g}^{(x)}}$ to the Riemannian measure $\nu_{\widehat{\mathfrak{g}}}$ on $\mathbb{S}^{d-1}$.
		If $\Phi: \mathbb{R}^{d} \to T_x\mathsf{M}$ is an isometry, 
		then for the pushforward measure of $\nu_{\iota^\ast\mathfrak{g}^{(x)}}$ under $\Phi$ holds
		\begin{equation}\label{Eq: Uniform distributions on unit tangent spheres as pushforward measures}
			\Phi_\sharp(\nu_{\widehat{\mathfrak{g}}}) = \nu_{\iota^\ast\mathfrak{g}^{(x)}}.
		\end{equation}
		Therefore
		\[
		\sigma^{(x)} := \frac{1}{\upomega_{d-1}} \nu_{\iota^\ast\mathfrak{g}^{(x)}}, \qquad x \in \mathsf{M},
		\]
		are probability measures on the unit tangent spheres for all $x \in \mathsf{M}$.
	\item \textbf{Geodesics:} Curves $\gamma : \mathsf{I} \to \mathsf{M}$, from an interval $\mathsf{I} \subseteq \mathbb{R}$ to the 
		Riemannian manifold $\mathsf{M}$, where the covariant derivative $\smash{\frac{\mathrm{D}}{\d t}}$ of the velocity vector field $\smash{\frac{\d \gamma_{(x,v)}}{\d t}}$ is zero everywhere are called geodesics.
		That is, geodesics are curves that satisfy $\smash{\frac{\mathrm{D}}{\d t}\frac{\d \gamma_{(x,v)}}{\d t}(s) = 0}$ for all $s \in \mathsf{I}$ or phrased more intuitively, geodesics are the curves of constant velocity.
		If every geodesic can be extended to have domain $\mathbb{R}$, the manifold $\mathsf{M}$ is called geodesically complete.
		Then for all $(x,v) \in T\mathsf{M}$ there exists a unique geodesic $\gamma_{(x,v)}: \mathbb{R} \to \mathsf{M}$ with $\gamma_{(x,v)}(0)= x$ and $\smash{\frac{\d \gamma_{(x,v)}}{\d t}(0)= v}$.
		Expressed in terms of the exponential map $\mathrm{Exp}: T\mathsf{M} \to \mathsf{M}$, this means $\gamma_{(x,v)}(\theta) = \mathrm{Exp}(x,\theta v)$ for all $\theta \in \mathbb{R}$.
		If a segment of a geodesic is the shortest path between its two endpoints, we call it minimizing.
		In more precise words, 
		for $s > 0$ the restriction $\gamma_{(x,v)}\vert_{[0,s]}$ of the geodesic $\gamma_{(x,v)}$ to the interval $[0,s]$ is minimizing if $\mathrm{dist}(x, \gamma_{(x,v)}(s)) = s$.
		Locally, all geodesics are minimizing.
	\item \textbf{Cut times:} Let $\mathsf{M}$ be a connected, geodesically complete Riemannian manifold.
		For $(x,v) \in T\mathsf{M}$ the associated cut time is defined as $t_{\mathrm{cut}}(x,v)\linebreak := \sup\{s \in (0, \infty) \mid \gamma_{(x,v)}\vert_{[0,s]} \text{ is minimizing}\}$.
		From this we can derive the cut locus $\mathrm{Cut}(x) := \{ \gamma_{(x,v)}(t_{\mathrm{cut}}(x,v)) \in \mathsf{M} \mid v \in U_x\mathsf{M}, t_{\mathrm{cut}}(x,v) < \infty\}$ at $x \in \mathsf{M}$.
		Note that cut loci are $\nu_{\mathfrak{g}}$-null sets.
		Moreover, for all $x \in \mathsf{M}$ the restriction of the exponential map $\mathrm{Exp}$ to the set $\{sv \in T_x\mathsf{M}\mid v \in U_x\mathsf{M}, s < t_{\mathrm{cut}}(x,v) \}$ is a diffeomorphism onto $\mathsf{M} \setminus \mathrm{Cut}(x)$.
		Through concatenation of these restricted exponential maps with isometries from $\mathbb{R}^d$ to the respective tangent spaces,
		one can obtain coordinate neighbourhoods of $\mathsf{M}$ that are adapted to the geodesics through a given point $x \in \mathsf{M}$.
		They are called normal coordinate neighbourhoods and we use the notation $(\mathsf{M}\setminus \mathrm{Cut}(x), \varphi_x)$.\\
		The infimum over all cut times is called injectivity radius $\mathrm{inj}(\mathsf{M}) := \inf_{(x,v) \in U\mathsf{M}} t_{\mathrm{cut}}(x,v)$.
		Observe that the supremum over all cut times is equal to the diameter of the manifold.
	\item \textbf{Ricci curvature:} We denote the curvature of the Riemannian manifold $(\mathsf{M}, \mathfrak{g})$ as $\mathcal{R}$.
		Then for all vector fields $X,Y$ and $Z$ on $\mathsf{M}$ also $\mathcal{R}(X,Y)Z$ is a vector field on $\mathsf{M}$.
		In fact, the value of the vector field $(\mathcal{R}(X,Y)Z)_x \in T_x\mathsf{M}$ at some point $x \in \mathsf{M}$ does not depend on the whole vector fields $X$, $Y$ and $Z$, but only on their values at the point $x$.
		Moreover, this quantity is linear in all three components.
		Therefore the Ricci curvature
		\[
			\mathrm{Ric}: U\mathsf{M} \to \mathbb{R},\qquad (x,v) \mapsto \mathrm{tr}\big(T_x\mathsf{M}\to T_x\mathsf{M},\ z \mapsto \mathcal{R}(z, v) v\big),
		\]
		where $\mathrm{tr}$ denotes the trace of a linear map, is well defined.
		We may understand the Ricci curvature as a summary quantity of the curvature.
	\item \textbf{Open submanifolds:} Any open subset $\mathsf{U}$ of a Riemannian manifold $(\mathsf{M}, \mathfrak{g})$ 
		is an embedded submanifold of $\mathsf{M}$.
		As such it inherits a Riemannian structure from the ambient Riemannian manifold $\mathsf{M}$ and thus becomes a Riemannian manifold itself.
		By virtue of the differential of the inclusion map $\mathsf{U} \to \mathsf{M}$, 
		we may consider its tangent bundle $T\mathsf{U}$ as a subset of $T\mathsf{M}$.
\end{enumerate}

\section{Main result}\label{S: main result}
Let $(\mathsf{M}, \mathfrak{g})$ be a connected, geodesically complete Riemannian manifold.
We consider target distributions on $\mathsf{M}$ that are absolutely continuous with respect to the Riemannian measure $\nu_{\mathfrak{g}}$ induced by $\mathfrak{g}$ on $\mathsf{M}$.
From such target distributions we can approximately sample with the geodesic slice sampler \citep{durmus2023geodesic}.
More precisely, let $p: \mathsf{M} \to [0, \infty)$ be a lower semi-continuous\footnote{In particular, the lower semi-continuity of $p$ ensures that $\mathsf{W}$ is an open submanifold of $\mathsf{M}$.} function such that $\mathsf{W}:= \{x \in \mathsf{M} \mid p(x)> 0\}$ is connected and $Z:= \int_{\mathsf{M}} p \ \d \nu_{\mathfrak{g}} \in (0, \infty)$.
The target distribution $\pi$ is then defined as 
\begin{equation}\label{Eq: target density}
	\pi(\d x) := \frac{1}{Z} p(x) \nu_{\mathfrak{g}}(\d x).
\end{equation}
For the sake of remaining self-contained, we provide a definition of the transition kernel of the geodesic slice sampler targeting the distribution $\pi$. 
However, for further background we refer to \citep{durmus2023geodesic} and \cite[Section 5.3]{hasenpflug2024slice}.
Geodesic slice sampling uses the superlevel sets
\[
	\mathsf{L}(t) := \{ x \in \mathsf{M} \mid p(x)> t\}, \qquad t \in (0, \infty)
\]
and the geodesic superlevel sets
\[
	\mathsf{L}(x,v,t) := \{ \theta \in \mathbb{R} \mid \gamma_{(x,v)}(\theta) \in \mathsf{L}(t)\}, \qquad (x,v) \in U\mathsf{M}, t \in (0, \infty).
\]
Moreover, for $(x,v) \in U\mathsf{M}$ define by 
\[
	\mathsf{W}(x,v) := \{ \theta \in \mathbb{R} \mid \gamma_{(x,v)}(\theta) \in \mathsf{W}\} = \bigcup_{t \in (0, \infty)} \mathsf{L}(x,v,t)
\]
the segment of the associated geodesic where $p$ is strictly positive.
Fix hyperparameters $m$ and $w$ that satisfy the following assumption.
\begin{assumption}\label{A: hyperparameters}
	Let $w \in (0, \infty)$.
	If 
	\[
		\uplambda := \sup_{(x,v) \in U\mathsf{M}}\mathrm{diam}\big(\mathsf{W}(x,v)\big) < \infty,
	\]
	then let $m \in \mathbb{N} \cup \{\infty\}$.
	Otherwise let $m \in \mathbb{N}$.
\end{assumption}
Observe that since $\mathsf{M}$ may have a finite injectivity radius, 
$\mathrm{diam}(\mathsf{W}) < \infty$ does not automatically imply $\uplambda < \infty$. 
Indeed, if the ambient manifold $\mathsf{M}$ itself is compact, we expect that in this case $\uplambda$ is usually infinite.
For simplicity, we do not express any dependencies on the hyperparameters $m$ and $w$ in the subsequent notation.

These hyperparameters are needed to specify a building block of the geodesic slice sampler: 
The stepping-out distributions $\xi_{\mathsf{L}(x,v,t)}$  are probability distributions on $\mathbb{R}^2$ for all $(x,v) \in U\mathsf{M}$, $t \in (0, p(x))$.
They describe the distribution of a random interval, that is, 
the first coordinate of $\mathbb{R}^2$ is interpreted as the left boundary of an interval in $\mathbb{R}$ and the second coordinate as the right boundary.
For more details we refer to Appendix\footnote{For brevity we drop here the dependence of the stepping-out distribution on its starting point, because when used within the geodesic slice sampler it is always $\theta = 0$.} \ref{S: stepping-out distribution}.

To generate a point from the intersection of the output of the stepping-out procedure and a given geodesic superlevel set,
the geodesic slice sampler then employs the shrinkage procedure.
We express the distribution of its output\footnote{More generally the shrinkage procedure defines a transition kernel. For brevity we use the shorthand notation $Q_{\mathsf{S}}^{\ell,r}(\cdot) := Q_{\mathsf{S}}^{\ell,r} (0,\cdot)$.}
as $Q_{\mathsf{L}(x,v,t)}^{\ell,r}$ for $(x,v) \in U\mathsf{M}$, $t \in (0, p(x))$ and $\ell,r \in \mathbb{R}$ with $(\ell,r) \cap \mathsf{L}(x,v,t) \neq \emptyset$, where $(\ell,r)$ corresponds to the output of the stepping-out procedure.
For more details we refer to Appendix \ref{S: shrinkage kernel}.

With this notation at hand, the transition kernel of the geodesic slice sampler is then given by
\begin{align*}
	K: \mathsf{W} \times \mathcal{B}(\mathsf{W}) &\to [0,1],\\
	(x, \mathsf{A}) &\mapsto \begin{multlined}[t]
		\frac{1}{p(x)} \int_{(0, p(x))} \int_{U_x\mathsf{M}} \int_{\mathbb{R}^2} \int_{\mathsf{L}(x,v,t) \cap (\ell,r)} \mathbbm{1}_{\mathsf{A}}\big(\gamma_{(x,v)}(\theta)\big)\\
			\times Q^{\ell,r}_{\mathsf{L}(x,v,t)}(\d \theta)\, \xi_{\mathsf{L}(x,v,t)}(\d (\ell,r))\, \sigma^{(x)}(\d v)\, \mathrm{Leb}_{1}(\d t).
	\end{multlined}
\end{align*}
We formulate our main result, which establishes the uniform ergodicity of this kernel in a compact setting.
\begin{theorem}\label{Thm: uniform ergodicity}
	Let $(\mathsf{M}, \mathfrak{g})$ be a geodesically complete, connected Riemannian manifold.
	Assume that $\pi$ is defined as in \eqref{Eq: target density}
	with unnormalised target density $p: \mathsf{M}\to [0, \infty)$ satisfying 
	\begin{itemize}
		\item $p$ is lower semicontinuous and $\int_{\mathsf{M}} p(x) \ \nu_{\mathfrak{g}}(\d x) < \infty$,
		\item $\|p\|_{\infty}< \infty$,
		\item $\mathsf{W}:= \{x \in \mathsf{M} \mid p(x)> 0\}$ is connected and relatively compact.
	\end{itemize}
	Fix hyperparameters $w \in (0, \infty)$ and $m \in \mathbb{N} \cup \{\infty\}$ such that Assumption \ref{A: hyperparameters} holds and such that there exists an $\upvarepsilon > 0$ satisfying
	\begin{equation}\label{Eq: assumption on stepping-out covering probability}
		\xi_{\mathsf{L}(x,v,t)}\left(\{(\ell, r) \in \mathbb{R}^2\ \mid\ [0, t_{\mathrm{cut}}(x,v)) \cap \mathsf{L}(x,v,t) \subseteq (\ell, r)\}\right)\geq \upvarepsilon
	\end{equation}
	for all $(x,v) \in U\mathsf{W}$, $t \in (0, p(x))$.
	Set 
	\begin{equation}\label{Eq: ergodicity constant}
		\uprho := 
			1 - \frac{\upvarepsilon}{\min\{mw , \uplambda\}} \cdot \frac{1}{\upkappa \upomega_{d-1}} \cdot \frac{\sup_{t \in (0, \infty)} t \nu_{\mathfrak{g}}(\mathsf{L}(t))}{\|p\|_{\infty}}
			\quad \in [0,1),
	\end{equation}
	where the constant $\upkappa \in (0, \infty)$, which is explicitly stated in \eqref{Eq: kappa}, only depends on properties of $\mathsf{M}$ and $\mathsf{W}$.
	Then we have for all $n \in \mathbb{N}$ that
	\begin{equation}\label{Eq: uniform ergodicity statement}
		\sup_{x \in \mathsf{W}} d_{\mathrm{tv}}\big(K^n(x, \cdot), \pi\big) \leq \uprho^n.
	\end{equation}
\end{theorem}

\begin{proof}
	To establish uniform ergodicity, we show that the whole state space $\mathsf{W}$ is small for $K$.
	To this end set
	\[
		\mathsf{R}_{x,v,t}:= \{(\ell, r) \in \mathbb{R}^2\ \mid\ [0, t_{\mathrm{cut}}(x,v)) \cap \mathsf{L}(x,v,t) \subseteq (\ell, r)\}
	\]
	for $ (x,v) \in U\mathsf{W}$ and $t \in (0, p(x))$. Fix $\mathsf{A} \in \mathcal{B}(\mathsf{W})$.
	We establish the necessary lower bound working from the inside out.
	By construction we have for random variables $(\boldsymbol{L}_{\mathsf{L}(x,v,t)},\boldsymbol{R}_{\mathsf{L}(x,v,t)})\sim \xi_{\mathsf{L}(x,v,t)}$ that
	\[
		\boldsymbol{R}_{\mathsf{L}(x,v,t)} - \boldsymbol{L}_{\mathsf{L}(x,v,t)} = \left( \mathfrak{T}_{\mathsf{L}(x,v,t)}+ \tau_{\mathsf{L}(x,v,t)} - 1\right) w
		\leq mw, \qquad \mathbb{P}\text{-almost surely},
	\]
	where the stopping times $\tau_{\mathsf{L}(x,v,t)}$ and $\mathfrak{T}_{\mathsf{L}(x,v,t)}$ are defined in \eqref{Eq: stepping-out stopping times}, and the right hand side may be infinity if $m = \infty$.
	Combining this with \eqref{Eq: assumption on stepping-out covering probability} and Lemma \ref{L: lower bound reeled shrinkage kernel}, we get for all $(x,v) \in U\mathsf{W}$ and $t \in (0, p(x))$ that
	\begin{align*}
		&\int_{\mathbb{R}^2} \int_{\mathsf{L}(x,v,t) \cap (\ell,r)} \mathbbm{1}_{\mathsf{A}}\big(\gamma_{(x,v)}(\theta)\big)\ Q^{\ell,r}_{\mathsf{L}(x,v,t)}(\d \theta)\, \xi_{\mathsf{L}(x,v,t)}(\d (\ell,r))\\
		&\qquad\geq \begin{multlined}[t]
			 \int_{\mathbb{R}^2} \frac{1}{\min\{r - \ell,\mathrm{diam}(\mathsf{L}(x,v,t))\}} \int_{\mathsf{L}(x,v,t) \cap (\ell,r)} \mathbbm{1}_{\mathsf{A}}\big(\gamma_{(x,v)}(\theta)\big)\\
			 \times \mathrm{Leb}_{1}(\d \theta) \, \xi_{\mathsf{L}(x,v,t)}(\d (\ell,r))
		\end{multlined}\\
		& \qquad \geq \int_{\mathsf{R}_{x,v,t}} \frac{1}{\min\{mw,\uplambda\}} \int_{\mathsf{L}(x,v,t) \cap (\ell,r)} \mathbbm{1}_{\mathsf{A}}\big(\gamma_{(x,v)}(\theta)\big)\ \mathrm{Leb}_{1}(\d \theta) \, \xi_{\mathsf{L}(x,v,t)}(\d (\ell,r))\\
		& \qquad \geq \begin{multlined}[t]
			\int_{\mathsf{R}_{x,v,t}} \frac{1}{\min\{mw,\uplambda\}} \int_{\mathsf{L}(x,v,t) \cap [0, t_{\mathrm{cut}}(x,v))} \mathbbm{1}_{\mathsf{A}}\big(\gamma_{(x,v)}(\theta)\big)\\
		\times \mathrm{Leb}_{1}(\d \theta) \, \xi_{\mathsf{L}(x,v,t)}(\d (\ell,r))
		\end{multlined}\\
		&\qquad = \frac{\xi_{\mathsf{L}(x,v,t)}(\mathsf{R}_{x,v,t})}{\min\{mw,\uplambda\}} \int_{\mathsf{L}(x,v,t) \cap [0, t_{\mathrm{cut}}(x,v))} \mathbbm{1}_{\mathsf{A}}\big(\gamma_{(x,v)}(\theta)\big)\ \mathrm{Leb}_{1}(\d \theta) \\
		& \qquad \geq \frac{\upvarepsilon}{\min\{mw,\uplambda\}} \int_{\mathsf{L}(x,v,t) \cap [0, t_{\mathrm{cut}}(x,v))} \mathbbm{1}_{\mathsf{A}}\big(\gamma_{(x,v)}(\theta)\big)\ \mathrm{Leb}_{1}(\d \theta).
	\end{align*}
	Next, we establish a suitable lower bound on the Ricci curvature to handle the integral over the uniform distribution on $U_x\mathsf{M}$.
	To this end, set
	\[
		\mathsf{W}_c := \bigcup_{(x,v) \in U\mathsf{W}} \big\{\mathrm{Exp}_x(\alpha v) \mid \alpha \in \mathrm{conv} \big(\mathsf{W}(x,v) \cap [0, t_{\mathrm{cut}}(x,v))\big)\big\}.
	\]
	Since $\mathsf{W}_c \subseteq B_{2 \mathrm{diam}(\mathsf{W})}(x_0)$ for some $x_0 \in \mathsf{W}$,
	we know that $\overline{\mathsf{W}}_c$ is compact as a closed subset of the compact set $\overline{B_{2 \mathrm{diam}(\mathsf{W})}(x_0)}$, see \cite[Theorem III.1.1]{sakai1996riemannian}.
	By virtue of \cite[Lemma 4.84]{lee2011introduction} and  \cite[Proposition 2.8]{lee2018introduction}, 
	this implies due to the paracompactness of manifolds (see \cite[Lemma 4.77]{lee2011introduction}) that also the set
	\[
		U\mathsf{W}_c:= \{(x,v) \in U\mathsf{M} \mid x \in \overline{\mathsf{W}}_c\}
	\]
	is compact.
	As $\mathrm{Ric}$ is a continuous function on the unit tangent bundle, see \cite[p.\ 45]{sakai1996riemannian},
	therefore there  exists a constant $ \upzeta \in \mathbb{R}$ such that 
	\[
		\min_{(x,v)\in U\mathsf{W}_c}\mathrm{Ric}(x,v) \geq (d-1)\upzeta.
	\]
	Set
	\[
		s_\upzeta^{d-1}(\alpha) := \begin{dcases}
									\upzeta^{(1-d)/2} \sin^{d-1}(\sqrt{\upzeta} \alpha ) , & \upzeta > 0,\\
									\alpha^{d-1}, & \upzeta = 0,\\
									 |\upzeta|^{(1-d)/2} \sinh^{d-1}(\sqrt{|\upzeta|}\alpha), & \upzeta < 0,
		\end{dcases}
		\qquad\alpha \in \mathbb{R}.
	\]
	Due to the lower bound on the Ricci curvature, the Bishop-Gromov volume comparison theorem tells us that for all $(x,v) \in U\mathsf{W}$ holds
	\begin{equation}\label{Eq: Bishop-Gromov volume comparison}
		s_\upzeta^{d-1}(\alpha) \geq \alpha^{d-1} \sqrt{\det(\mathfrak{g}, \varphi_x)}\big(\mathrm{Exp}_x(\alpha v)\big), \quad \alpha \in \mathrm{conv}\big([0, t_{\mathrm{cut}}(x,v)) \cap \mathsf{W}(x,v)\big),
	\end{equation}
	where $(\mathsf{M} \setminus \mathrm{Cut}(x), \varphi_x)$ denotes a normal coordinate neighbourhood for each $x \in \mathsf{M}$, see \cite[Theorem IV.3.1(2)]{sakai1996riemannian}.
	Observe that for the case $\upzeta > 0$ a priori the above inequality only holds for $\alpha < \uppi/\sqrt{\upzeta}$.
	However, if there would exist $(x,v) \in U\mathsf{W}$ such that $\uppi/\sqrt{\upzeta} \in \mathrm{conv}([0, t_{\mathrm{cut}}(x,v)) \cap \mathsf{W}(x,v))$ then continuity  would imply
	\begin{align*}
		0 &\leq \left(\frac{\uppi}{\sqrt{\upzeta}}\right)^{d-1} \sqrt{\det(\mathfrak{g}, \varphi_x)}\left (\mathrm{Exp}_x\left (\frac{\uppi}{\sqrt{\upzeta}} v\right )\right )
%		\\ &= \lim_{\alpha \to \uppi/\sqrt{\upzeta}} \alpha^{d-1} \sqrt{\det(\mathfrak{g}, \varphi_x)}\big(\mathrm{Exp}_x(\alpha v)\big)
%		\leq \lim_{\alpha \to \uppi/\sqrt{\upzeta}} s_\upzeta^{d-1}(\alpha)=
		\leq s_\upzeta^{d-1}\left(\frac{\uppi}{\sqrt{\upzeta}}\right) 
%		= \sin(\uppi)^{d-1} 
		= 0,
	\end{align*}
	as $\uppi/\sqrt{\upzeta}$ can be approximated from below by elements in $\mathrm{conv}([0, t_{\mathrm{cut}}(x,v)) \cap \mathsf{W}(x,v))$.
	Hence, we have $\sqrt{\det(\mathfrak{g}, \varphi_x)}(y)= 0$ for $y= \mathrm{Exp}_x(\uppi/ \sqrt{\upzeta} v) \in \mathsf{M} \setminus \mathrm{Cut}(x)$.
	This contradicts the fact that $\sqrt{\det(\mathfrak{g}, \varphi_x)}$ is strictly positive on $\mathsf{M} \setminus \mathrm{Cut}(x)$.
	Therefore if $\upzeta > 0$, for all $(x,v) \in U\mathsf{W}$ holds  $\mathrm{conv}([0, t_{\mathrm{cut}}(x,v)) \cap \mathsf{W}(x,v)) \subseteq [0,\uppi/\sqrt{\upzeta})$.
	In particular, with
	\begin{equation}\label{Eq: kappa}
		\upkappa := \begin{dcases}
					 \upzeta^{(1-d)/2} \sin^{d-1}\left ((\sqrt{\upzeta} \mathrm{diam}(\mathsf{W})) \wedge \frac{\uppi}{2} \right ), & \upzeta > 0,\\
					\mathrm{diam}(\mathsf{W})^{d-1}, & \upzeta = 0,\\
					|\upzeta|^{(1-d)/2} \sinh^{d-1}\left (\sqrt{|\upzeta|}\mathrm{diam}(\mathsf{W})\right ) , & \upzeta < 0,
			\end{dcases}
	\end{equation}
	we get from \eqref{Eq: Bishop-Gromov volume comparison} that
	\[
		\alpha^{d-1} \sqrt{\det(\mathfrak{g}, \varphi_x)}\big(\mathrm{Exp}_x(\alpha v)\big) \leq s_\upzeta^{d-1}(\alpha) \leq \upkappa, \qquad \alpha \in [0, t_{\mathrm{cut}}(x,v)) \cap \mathsf{W}(x,v).
	\]
	In combination with \eqref{Eq: Uniform distributions on unit tangent spheres as pushforward measures} and a transition from polar to Cartesian coordinates this yields for all $x \in \mathsf{W}$ and all $t \in (0, \infty)$ that 
	\begin{align*}
		&\int_{U_x\mathsf{M}}\int_{\mathsf{L}(x,v,t) \cap [0, t_{\mathrm{cut}}(x,v))} \mathbbm{1}_{\mathsf{A}}\big(\gamma_{(x,v)}(\theta)\big)\ \mathrm{Leb}_{1}(\d \theta)\, \sigma^{(x)}(\d v)\\
		&\qquad\geq \begin{multlined}[t]
			\frac{1}{\upkappa} \int_{U_x\mathsf{M}}\int_{\mathsf{L}(x,v,t) \cap [0, t_{\mathrm{cut}}(x,v))} \theta^{d-1} \sqrt{\det(\mathfrak{g}, \varphi_x)}\big(\mathrm{Exp}_x(\theta v)\big) \mathbbm{1}_{\mathsf{A}}\big(\gamma_{(x,v)}(\theta)\big)\\
			\times \mathrm{Leb}_{1}(\d \theta)\, \sigma^{(x)}(\d v)
		\end{multlined}\\
		&\qquad = \frac{1}{\upkappa \upomega_{d-1}} \int_{\varphi_x(\mathsf{L}(t) \cap \mathsf{A}\setminus \mathrm{Cut}(x))}  \sqrt{\det(\mathfrak{g}, \varphi_x)} \left(\varphi_x^{-1}(z)\right) \ \mathrm{Leb}_{d}(\d z)\\
		& \qquad = \frac{1}{\upkappa \upomega_{d-1}} \nu_{\mathfrak{g}}\big(\mathsf{L}(t) \cap \mathsf{A} \setminus \mathrm{Cut}(x)\big)
		= \frac{1}{\upkappa \upomega_{d-1}} \nu_{\mathfrak{g}}\big(\mathsf{L}(t) \cap \mathsf{A} \big),
	\end{align*}
	where the last equality holds because $\mathrm{Cut}(x)$ is a $\nu_{\mathfrak{g}}$-null set, see \cite[Theorem 10.34]{lee2018introduction}.
	Finally, observe that for all $x \in \mathsf{W}$ and all $s \in (0, \|p\|_{\infty})$ holds
	\begin{align*}
		&\frac{1}{p(x)} \int_{(0, p(x))} \nu_{\mathfrak{g}}\big(\mathsf{L}(t) \cap \mathsf{A} \big)\ \mathrm{Leb}_{1}(\d t)
		=\int_{\mathsf{A}} \frac{1}{p(x)} \int_{(0, p(x)) } \mathbbm{1}_{\mathsf{L}(t)}(y) \ \mathrm{Leb}_{1}(\d t)\, \nu_{\mathfrak{g}}(\d y)\\
		&\qquad= \int_{\mathsf{A}} \frac{p(x) \wedge p(y)}{p(x)} \ \nu_{\mathfrak{g}}(\d y)
		\geq \int_{\mathsf{A} \cap \mathsf{L}(s)}  \frac{p(x) \wedge p(y)}{p(x)} \ \nu_{\mathfrak{g}}(\d y) \\
		&\qquad\geq \frac{s}{\|p\|_{\infty}} \nu_{\mathfrak{g}}\big(\mathsf{A} \cap \mathsf{L}(s)\big).
	\end{align*}
	Overall we get for all $x \in \mathsf{W}$ and all $s \in (0, \|p\|_{\infty})$ that
	\begin{align*}
		&K(x, \mathsf{A}) 
		\geq 	\begin{multlined}[t]
			\frac{1}{p(x)} \int_{(0, p(x))} \frac{\upvarepsilon}{\min\{mw,\uplambda\}}\int_{U_x\mathsf{M}}  \int_{\mathsf{L}(x,v,t) \cap [0, t_{\mathrm{cut}}(x,v))} \mathbbm{1}_{\mathsf{A}}\big(\gamma_{(x,v)}(\theta)\big)\\
			\times \mathrm{Leb}_{1}(\d \theta)\, \sigma^{(x)}(\d v)\, \mathrm{Leb}_{1}(\d t)
		\end{multlined}\\
		&\qquad\geq  \frac{\upvarepsilon}{\min\{mw,\uplambda\}}\frac{1}{\upkappa \upomega_{d-1}}	\frac{1}{p(x)} \int_{(0, p(x))} \nu_{\mathfrak{g}}\big(\mathsf{L}(t) \cap \mathsf{A} \big)\ \mathrm{Leb}_{1}(\d t)\\
		&\qquad\geq  \frac{\upvarepsilon}{\min\{mw,\uplambda\}}\frac{1}{\upkappa \upomega_{d-1}}  \frac{s}{\|p\|_{\infty}} \nu_{\mathfrak{g}}\big(\mathsf{A} \cap \mathsf{L}(s)\big),
	\end{align*}
	i.e.\ that $\mathsf{W}$ is $\upvarepsilon s(\min\{mw,\uplambda\} \upkappa \upomega_{d-1} \|p\|_{\infty})^{-1} \nu_{\mathfrak{g}}\vert_{\mathsf{L}(s)}$-small.
	Consequently, by \cite[Theorem 16.2.4]{meyn2009markov} for all $n \in \mathbb{N}$ holds
	\[
		\sup_{x \in \mathsf{W}} d_{\mathrm{tv}}\big(K^n(x, \cdot), \pi\big) \leq \left(1 -  \frac{\upvarepsilon}{\min\{mw,\uplambda\}}\frac{1}{\upkappa \upomega_{d-1}}  \frac{s \nu_{\mathfrak{g}}(\mathsf{L}(s))}{\|p\|_{\infty}}\right)^n.
	\]
 	Taking the infimum over $s$ yields the desired statement.
\end{proof} 
It is also possible to replace \eqref{Eq: assumption on stepping-out covering probability} by a direct condition on the hyperparameters at the expense of a possibly worse estimate of the convergence speed.
To this end define
\[
	\Updelta := \sup_{(x,v) \in U\mathsf{\mathsf{W}}, t \in (0, p(x))} \mathrm{Leb}_{1}\Big(\mathrm{conv}\big([0, t_{\mathrm{cut}}(x,v)) \cap \mathsf{L}(x,v,t) \big)\setminus \mathsf{L}(x,v,t)\Big).
\]
This constant is smaller than the diameter of $\mathsf{W}$ and may be understood as the size of the ``largest'' interruption of the geodesic level sets up to the corresponding cut time.
\begin{corollary}\label{C: uniform ergodicity}
	In the setting of Theorem \ref{Thm: uniform ergodicity} assume that instead of \eqref{Eq: assumption on stepping-out covering probability} we have 
	$\mathrm{diam}(\mathsf{W})m^{-1}\mathbbm{1}_{(0, \infty)}(m)< w- \Updelta \mathbbm{1}_{[2, \infty]}(m)$.
	Then \eqref{Eq: uniform ergodicity statement} holds for $\uprho$ given as in \eqref{Eq: ergodicity constant} with 
	\[
		\upvarepsilon = 1 - \frac{\mathrm{diam}(\mathsf{W})}{mw} \mathbbm{1}_{(0, \infty)}(m) - \frac{\Updelta}{w}\mathbbm{1}_{[2, \infty]}(m).
	\]
\end{corollary}
\begin{proof}
	For all $(x,v) \in U\mathsf{W}$, $t \in (0, p(x))$ we apply Lemma \ref{L: lower bound on stepping-out covering probability} with $\mathsf{S} = \mathsf{L}(x,v,t)$, $\theta = 0$ and $C= t_{\mathrm{cut}}(x,v)$.
	Taking the infimum yields that \eqref{Eq: assumption on stepping-out covering probability} holds for $\upvarepsilon = 1 - \mathrm{diam}(\mathsf{W})(mw)^{-1} \mathbbm{1}_{(0, \infty)}(m) - \Updelta w^{-1}\mathbbm{1}_{[2, \infty]}(m)$.
\end{proof}
A closer inspection of the constant $\uprho$ in Theorem \ref{Thm: uniform ergodicity} reveals that it depends on three different factors, which are the underlying state space $\mathsf{W}$ together with the ambient Riemannian manifold $\mathsf{M}$, the target distribution $\pi$, and the hyperparameters $m$ and $w$.
In the following, we analyse in more detail how these three factors influence $\uprho$.
\begin{figure}[hbt]
	\centering
	\begin{subfigure}{.47\textwidth}
		\flushleft
		\includegraphics[height=0.17\textheight]{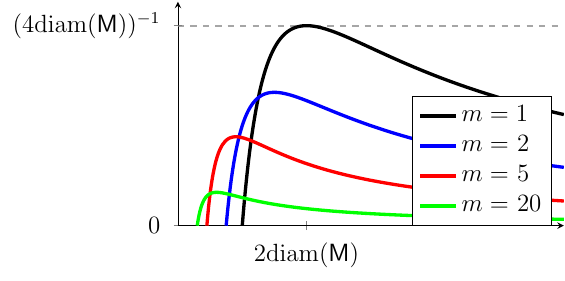}
		\caption{$\uplambda = \infty$.}
	\end{subfigure}
	\hfill
	\begin{subfigure}{.47\textwidth}
		\flushright
		\includegraphics[height=0.17\textheight]{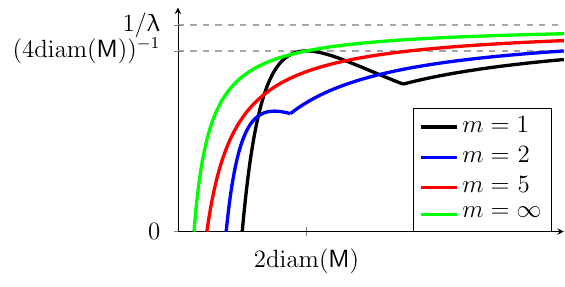}
		\caption{$2 \mathrm{diam}(\mathsf{M}) < \uplambda \leq 4 \mathrm{diam}(\mathsf{M})$.}
	\end{subfigure}
	
	\vspace*{1em}
	\begin{subfigure}{.47\textwidth}
		\flushleft
		\includegraphics[height=0.17\textheight]{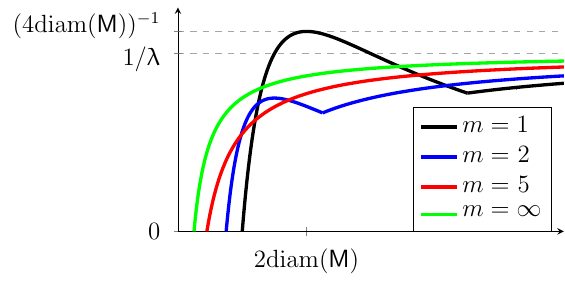}
		\caption{4 $\mathrm{diam}(\mathsf{M}) < \uplambda$.}
	\end{subfigure}
	\hfill
	\begin{subfigure}{.47\textwidth}
		\flushright
		\includegraphics[height=0.17\textheight]{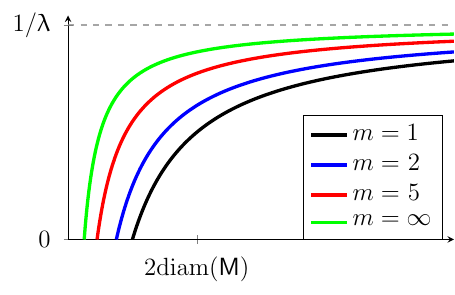}
		\caption{$\uplambda \leq 2 \mathrm{diam}(\mathsf{M})$.}
	\end{subfigure}
	\caption{Illustration of the different regimes for $q$ defined in \eqref{Eq: Ergodicity constant term depending on hyperparameters}. The horizontal axis shows the hyperparameter $w$, the vertical axis shows $q(m, \cdot)$, where different values of the hyperparameter $m$ correspond to different coloured graphs. All plots are drawn for $\Updelta > 0$. Note that $y$-axis scaling is not comparable between plots.}\label{F: sketches of q}
\end{figure}
\begin{remark}[Hyperparameters]
	The only part of $\uprho$ that depends on the hyperparameters $m$ and $w$ is the factor $\upvarepsilon/ \min\{mw, \uplambda\}$ in the subtrahend.
	Since the influence of $m$ and $w$ on the possible values of $\upvarepsilon$ is not straightforward,
	we restrict our analysis to the setting of Corollary \ref{C: uniform ergodicity} and work with the lower bound on $\upvarepsilon$ provided there,
	that is, 
	we investigate for
	\begin{equation}\label{Eq: Ergodicity constant term depending on hyperparameters}
		\begin{multlined}[0.9\textwidth]
			q(m,w) := \frac{\left(1 - \frac{\mathrm{diam}(\mathsf{W})}{mw} \mathbbm{1}_{(0, \infty)}(m) - \frac{\Updelta}{w}\mathbbm{1}_{[2, \infty]}(m)\right)}{\min\{mw, \uplambda\}},\\ m \in \mathbb{N} \cup \{\infty\}, w \in (0, \infty),
		\end{multlined}
	\end{equation}
	the question of an optimal choice of the hyperparameters with respect to the convergence guarantee provided by Corollary \ref{C: uniform ergodicity},
	i.e.\ for a choice that maximises $q$.
	At this point, let us emphasize that any analysis of $q$ is only an analysis of an upper bound for the convergence of $\sup_{x \in \mathsf{W}} d_{\mathrm{tv}}(K^n(x, \cdot), \pi)$, not of the true convergence behaviour itself. 
	Moreover, we completely omit cost per Markov chain iteration in our subsequent considerations.  
	
	Roughly speaking, the optimal choice depends on the value of $\uplambda$.
	If $\uplambda$ is infinite,
	then $m = 1$ and $w = 2 \mathrm{diam}(\mathsf{M})$ gives the fastest convergence guarantee in Corollary \ref{C: uniform ergodicity}.
	On the other hand, if $\uplambda < \infty$, then, except for some (pathological) cases, increasing the value of $w$ improves the convergence guarantee for every fixed value of $m$.
	A noteworthy exception is here the case if also $\Updelta = 0$.
	In this case the convergence guarantee is the fastest if $m = \infty$ and the value of $w\in (0, \infty)$ has no effect.
	
	We discuss the two different settings in more detail. 
	If $\uplambda = \infty$, then we may only choose $m \in \mathbb{N}$.
	An optimisation over $w$ for fixed values of $m$, 
	yields that for all $m \in \mathbb{N}$ the function $q(m, \cdot)$ attains its global maximum at $w_m^\ast = 2\mathrm{diam}(\mathsf{W})m^{-1} + 2\Updelta\mathbbm{1}_{[2, \infty)}(m)$.
	Hence, we get 
	\[
		\max_{w\in (0, \infty), m \in \mathbb{N}} q(w,m) = \frac{1}{4 \mathrm{diam}(\mathsf{W})},
	\]
	which is attained for $m=1$ and $w = 2 \mathrm{diam}(\mathsf{W})$ if $\Updelta > 0$, and otherwise for $mw = 2 \mathrm{diam}(\mathsf{W})$.
	
	We now turn to the case $\uplambda <  \infty$.
	First of all note that $\lim_{w \to \infty} q(m,w) = \uplambda^{-1}$ for all $m \in \mathbb{N} \cup \{\infty\}$.
	Moreover, $q(\infty, \cdot)$ is everywhere increasing, and for  $m \in \mathbb{N}$ the function $q(m, \cdot)$ is increasing as soon as $w \geq \uplambda/m$ or $w \leq w_m^\ast$.
	If $\Updelta > 0$ and $\uplambda$ is big compared to the diameter of $\mathsf{W}$, more precisely, if $\uplambda > 2 \mathrm{diam}(\mathsf{W})$, then for 
	small $m \in \mathbb{N}$
	the function $q(m, \cdot)$ has a local maximum at $w_m^\ast$ and a local minimum at $w= \uplambda/m$.
	However, only if $\uplambda > 4 \mathrm{diam}(\mathsf{W})$, the global maximal point of $q$ lies at $m=1$ and $w = 2\mathrm{diam}(\mathsf{W})$, taking again the value $(4 \mathrm{diam}(\mathsf{W}))^{-1}$.
	If $\Updelta = 0$, then we have almost the same picture, except that for $\uplambda > 2 \mathrm{diam}(\mathsf{W})$ the function $q(m, \cdot)$ has a local maximum at $w_m^\ast$ for all $m \in \mathbb{N}$,
	and if $\uplambda > 4 \mathrm{diam}(\mathsf{W})$ the set of maximal points consists of all $m \in \mathbb{N}$ and $ w \in (0, \infty)$ with $mw = 2 \mathrm{diam}(\mathsf{W})$.
	We turn to the last case of $\uplambda \leq 2 \mathrm{diam}(\mathsf{W})$.
	Then $q(m, \cdot)$ is always increasing in $w$ for any fixed $m \in \mathbb{N}\cup \{\infty\}$.
	More precisely, if $\Updelta > 0$, this increase is strict.
	However, if $\Updelta = 0$, then the maximum $\uplambda^{-1}$ of $q$ is attained for $m= \infty$ and any value for $w$.
	See Figure \ref{F: sketches of q} for an illustration of these four different scenarios.
	
	One additional thing to observe from  Figure \ref{F: sketches of q} is that as long as $w$ is chosen ``large enough'' $q$ is not too sensitive to the exact choice of $w$.
\end{remark}
\begin{remark}[Target distribution]\label{R: Dependence of ergodicity constant on target distirbution}
	The unnormalised target density $p$ enters the constant $\uprho$ in two different places.
	Firstly, it interacts with the hyperparameters $m$ and $w$ affecting for which values of $\upvarepsilon$ condition \eqref{Eq: assumption on stepping-out covering probability} holds.
	However, at least in the setting of Corollary \ref{C: uniform ergodicity}, the influence at this point may be regarded as limited in the sense that,
	 with the estimate $\Updelta \leq \mathrm{diam}(\mathsf{W})$,
	the unnormalised target density $p$ may be completely eliminated from the first factor of the subtrahend in $\uprho$.
	The other place where $p$ enters $\uprho$ is the factor $\sup_{t \in (0, \infty)} t \nu_{\mathfrak{g}}(\mathsf{L}(t))/\|p\|_\infty^{-1}$.
	Due to the lower semi-continuity of $p$ we may rewrite this term completely in terms of the level set function
	\[
		\mathcal{L}: (0, \infty) \to [0, \infty), \qquad t \mapsto \nu_{\mathfrak{g}}\big(\mathsf{L}(t)\big).
	\]
	Namely, we have
	\[
		\frac{\sup_{t \in (0, \infty)} t \nu_{\mathfrak{g}}(\mathsf{L}(t))}{\|p\|_\infty} = \frac{\sup_{t \in (0, \infty)} t \mathcal{L}(t)}{\sup \{t \in (0, \infty) \mid \mathcal{L}(t) > 0\}},
	\]
	such that this factor does not directly depend on the unnormalised target density $p$, but rather only on its level set function $\mathcal{L}$.
	In the literature,  the level set function appears at several points in the role of controlling the convergence behaviour of (ideal) slice samplers, see e.g.\ \citep{roberts1999convergence,mira2002efficiency, natarovskii2021quantitative,schaer2023wasserstein, rudolf2024dimension}. 
	In particular, it is known that the spectral gap of two ideal slice samplers coincides if they target distributions with the same level set function \citep[Theorem 3.8]{schaer2024slice}.
\end{remark}
\begin{remark}[State space and ambient Riemannian manifold]
	It remains to analyse the second factor in the subtrahend of $\uprho$, which only depends on the state space $\mathsf{W}$ and the ambient Riemannian manifold $\mathsf{M}$.
	However, from the previous two remarks we know that both $\mathsf{W}$ and $\mathsf{M}$ enter also the first and the third factor in $\uprho$'s subtrahend,
	such that within this remark we work with the following reformulation 
	\[
		\uprho = 
		1 - \frac{\upvarepsilon\ \mathrm{diam}(\mathsf{W})}{\min\{mw , \uplambda\}} \cdot \frac{\nu_{\mathfrak{g}}(\mathsf{W})}{\mathrm{diam}(\mathsf{W})\upkappa \upomega_{d-1}} \cdot \frac{\sup_{t \in (0, \infty)} t \nu_{\mathfrak{g}}(\mathsf{L}(t))}{\|p\|_{\infty}\ \nu_{\mathfrak{g}}(\mathsf{W})},
	\]
	which to some extend renormalises the first and the third factor with respect to the properties of $\mathsf{W}$ and $\mathsf{M}$.
	If $\mathsf{M} = \mathsf{W}$, we can use the Berger isoembolic inequality, see \cite[Theorem VI.2.1]{sakai1996riemannian}, 
	to get a more explicit expression in the dimension and the curvature:
	Combining the bound $\nu_{\mathfrak{g}}(\mathsf{M}) \geq \mathrm{inj}(\mathsf{M})^d/\uppi^d \upomega_d$ from the Berger isoembolic inequality with $\upomega_d/\upomega_{d-1} \geq \sqrt{2\uppi/d}$, which holds by virtue of \cite[Lemma 6]{mathe2007simple},
	we get the estimate
	\begin{equation}\label{Eq: estimate of manifold part of ergodicity constant}
			\begin{aligned}
			\frac{\nu_{\mathfrak{g}}(\mathsf{M})}{\mathrm{diam}(\mathsf{M})\upkappa \upomega_{d-1}}
			&\geq \left(\frac{\mathrm{inj}(\mathsf{M})}{\uppi}\right)^d \sqrt{\frac{2\uppi}{d}} \frac{1}{\mathrm{diam}(\mathsf{M}) \kappa}\\
			&\geq \begin{dcases}
				\sqrt{\frac{2}{\uppi}} \frac{1}{\sqrt{d}} \left(\frac{\mathrm{inj}(\mathsf{M}) \sqrt{\upzeta}}{\uppi}\right)^d,& \upzeta > 0,\\
				\sqrt{2\uppi} \frac{1}{\sqrt{d}} \left(\frac{\mathrm{inj}(\mathsf{M}) }{\uppi\, \mathrm{diam}(\mathsf{M})}\right)^d,& \upzeta = 0,\\
				\sqrt{2\uppi} \frac{1}{\sqrt{d}} \left(\frac{\mathrm{inj}(\mathsf{M}) \sqrt{|\upzeta|}}{\uppi\, \sinh(\sqrt{|\upzeta|}\mathrm{diam}(\mathsf{M}))}\right)^d,& \upzeta < 0,
			\end{dcases}
		\end{aligned}
	\end{equation}
	where in the case $\upzeta > 0$ we also used that $\sqrt{\upzeta}\mathrm{diam}(\mathsf{M}) \leq \uppi$ by the Cheng maximal diameter theorem, see \cite[Theorem IV.3.5]{sakai1996riemannian}.
	As  $\mathrm{inj}(\mathsf{M}) \leq \mathrm{diam}(\mathsf{M})$, 
	the dependence on the dimension of the right hand side in \eqref{Eq: estimate of manifold part of ergodicity constant}  is exponential,
	unless we have $\upzeta > 0$ and $\mathrm{inj}(\mathsf{M}) \sqrt{\upzeta} = \uppi$.
	However, again by virtue of the Cheng maximal diameter theorem this is the case if and only if $\mathsf{M}$ is isometric to a Euclidean sphere with radius $\upzeta^{-1/2}$.
\end{remark}
We conclude by considering the results of Theorem \ref{Thm: uniform ergodicity} and Corollary \ref{C: uniform ergodicity} in two explicit settings.
\begin{example}[Geodesic slice sampling on the sphere]
	As remarked in \citep{durmus2023geodesic},
	for distributions on the Euclidean unit sphere $\mathbb{S}^{d}$, i.e.\ if $\mathsf{M}=\mathsf{W}=\mathbb{S}^{d}$, 
	we recover the geodesic shrinkage slice sampler on the sphere proposed in \citep{habeck2023geodesic}
	if the hyperparameters are set as $m= 1$ and $w= 2\uppi$.
	In this case the random interval generated by the stepping-out procedure is always equivalent to one winding of the great circle.
	Therefore we get that \eqref{Eq: assumption on stepping-out covering probability} holds for $\upvarepsilon= 1$, 
	and Theorem \ref{Thm: uniform ergodicity} recovers the uniform ergodicity statement of \citep{habeck2023geodesic} up to the estimate $\upomega_d/\upomega_{d-1} \geq \sqrt{2\uppi/d}$.
	Observe that an application of Corollary \ref{C: uniform ergodicity} only yields the weaker statement that \eqref{Eq: assumption on stepping-out covering probability}, and consequently \eqref{Eq: uniform ergodicity statement}, holds for $\upvarepsilon= 1/2$.
	This demonstrates that Corollary \ref{C: uniform ergodicity} is not necessarily sharp.
	
	Finally, note that \cite{schaer2024dimension} show in the case of $\pi$ being the uniform distribution on $\mathbb{S}^d$ a dimension independent bound on the Wasserstein contraction coefficient of the geodesic slice sampler on the sphere,
	whereas our bound on the convergence in total variation distance is dimension dependent.
	Note that these two results concern different metrics.
\end{example}
\begin{example}[Hit-and-run algorithm]
	Let $\mathsf{C} \subseteq\mathbb{R}^d$ be the interior of a convex body.
	For $m= \infty$ and any $w \in (0, \infty)$, the geodesic slice sampler\footnote{The ambient Riemannian manifold is $\mathbb{R}^d$.} targeting the uniform distribution on $\mathsf{C}$ becomes the hit-and-run algorithm. 
	The convexity of $\mathsf{C}$ ensures that \eqref{Eq: assumption on stepping-out covering probability} is satisfied for $\upvarepsilon = 1$ under this choice of hyperparameters.
	We get that Theorem \ref{Thm: uniform ergodicity} holds with  $\uprho = 1 - \mathrm{Leb}_{d}(\mathsf{C})/ (\upomega_{d-1} \mathrm{diam}(\mathsf{C})^d)$.
	Up to a factor of two in the subtrahend, this coincides with the bound (implicitly) provided in \cite[Theorem 3]{diaconis1998on}.
\end{example}
 
\section*{Appendix.}
\appendix
\section{The stepping-out distribution}\label{S: stepping-out distribution}
This section contains some auxiliary statements about the stepping-out procedure.
To remain self-contained, we first recall the definition of the stepping-out distribution for finite and infinite hyperparameter $m$ from \citep{durmus2023geodesic} and \citep[Section 4.1.1]{schaer2024slice},
which describes the output of the stepping-out procedure proposed by \citet{neal2003slice}, before we turn to its properties.
To this end let $\mathsf{S} \in \mathcal{B}(\mathbb{R})$ and $\theta \in \mathbb{R}$.
Fix hyperparameters $m$ and $w$ that satisfy the following assumption.
\begin{assumption}\label{A: hyperparameters general}
	Let $w \in (0, \infty)$.
	If $\mathrm{diam}(\mathsf{S}) < \infty$, then let $m \in \mathbb{N} \cup \{\infty\}$.
	Otherwise let $m \in \mathbb{N}$.
\end{assumption}
For simplicity, we omit any dependencies on the hyperparameters $m$ and $w$ in the subsequent notation unless needed for clarity.

Let $\Upsilon \sim \mathrm{Unif}(0, w)$.
For the sake of technical simplicity we assume that $\Upsilon$ only takes values in $(0, w)$.
Set
\[
	L_i^{(\theta)} := \theta -\Upsilon - (i-1) w, \qquad R_i^{(\theta)}:= \theta-\Upsilon + iw, \qquad i \in \mathbb{N}.
\]
Moreover, define the stopping times 
\begin{equation}\label{Eq: stepping-out stopping times}
	\begin{aligned}
		\tau_{\mathsf{S}}^{(\theta)} &:= \begin{dcases}
			\inf\{i \in \mathbb{N} \mid L_i^{(\theta)} \notin \mathsf{S}\} \wedge J, & m < \infty,\\
			\inf\{i \in \mathbb{N} \mid L_i^{(\theta)} \notin \mathsf{S}\} , & m = \infty,
		\end{dcases}\\
		\mathfrak{T}_{\mathsf{S}}^{(\theta)} &:= \begin{dcases}
			\inf\{i \in \mathbb{N} \mid R_i^{(\theta)} \notin\mathsf{S}\} \wedge (m+1-J), & m < \infty,\\
			\inf\{i \in \mathbb{N} \mid R_i^{(\theta)} \notin\mathsf{S}\} , & m= \infty,
		\end{dcases} 
	\end{aligned}
\end{equation}
where $J \sim \mathrm{Unif}(\{1, \ldots, m\})$ is independent of all previous random variables.
Under Assumption \ref{A: hyperparameters general} these stopping times are almost surely finite, see \cite[Lemma 4.1]{schaer2024slice}, and we may define
\[ 
	\boldsymbol{L}_{\mathsf{S}}^{(\theta)}:= L_{\tau_{\mathsf{S}}^{(\theta)}}^{(\theta)}, \qquad 	\boldsymbol{R}_{\mathsf{S}}^{(\theta)}:= R_{\mathfrak{T}_{\mathsf{S}}^{(\theta)}}^{(\theta)}. 
\]
The stepping-out distribution $\xi_{\mathsf{S}}^{(\theta)}$ is then defined as the joint distribution of $\boldsymbol{L}_{\mathsf{S}}^{(\theta)} $ and $\boldsymbol{R}_{\mathsf{S}}^{(\theta)}$.

The following two lemmas guarantee that the stepping-out distribution with hyperparameter $m=\infty$ may be used within the geodesic slice sampler under Assumption \ref{A: hyperparameters} without harming the reversibility, positive semi-definiteness or ergodicity of the kernel of the geodesic slice sampler established in \citep{durmus2023geodesic} and \citep[Theorem 75]{hasenpflug2024slice}.
\begin{lemma}\label{L: Weak convergence stepping out}
	Let $\mathsf{S} \in \mathcal{B}(\mathbb{R})$ with $\mathrm{diam}(\mathsf{S}) < \infty$ and $\theta \in \mathbb{R}$.
	The sequence of measures $(\xi_{\mathsf{S}}^{(\theta,k)})_{k\in \mathbb{N}}$ converges weakly to $\xi_{\mathsf{S}}^{(\theta,\infty)}$,
	where the additional superscript index indicates the choice of the hyperparameter $m$.
\end{lemma}

\begin{proof}
	As in the statement, any additional superscript indices indicate the choice of the hyperparameter $m$.
	We show that $( \boldsymbol{L}_{\mathsf{S}}^{(\theta, k)},\boldsymbol{R}_{\mathsf{S}}^{(\theta,k)} )_{k \in \mathbb{N}}$ converges in probability to $( \boldsymbol{L}_{\mathsf{S}}^{(\theta, \infty)},\boldsymbol{R}_{\mathsf{S}}^{(\theta,\infty)} )$. This implies the desired weak convergence.
	
	Let $\varepsilon > 0$. Without loss of generality we may assume that $\varepsilon < \sqrt{2}w$.
	Note that for all $k \in \mathbb{N}$
	\begin{align*}
		\left |   \boldsymbol{L}_{\mathsf{S}}^{(\theta,k)} -  \boldsymbol{L}_{\mathsf{S}}^{(\theta, \infty)}\right |& =
		\begin{dcases}
			0, & \tau_{\mathsf{S}}^{(\theta,\infty)} \leq J^{(k)},\\
			(\tau_{\mathsf{S}}^{(\theta,\infty)} - J^{(k)})w \geq w, & \tau_{\mathsf{S}}^{(\theta,\infty)} > J^{(k)},
		\end{dcases} \\
		\left |   \boldsymbol{R}_{\mathsf{S}}^{(\theta,k)} -  \boldsymbol{R}_{\mathsf{S}}^{(\theta, \infty)}\right |& =
		\begin{dcases}
			0, & \mathfrak{T}_{\mathsf{S}}^{(\theta,\infty)} \leq k+1-J^{(k)},\\
			(\mathfrak{T}_{\mathsf{S}}^{(\theta,\infty)} - (k+1-J^{(k)}))w \geq w, & \mathfrak{T}_{\mathsf{S}}^{(\theta,\infty)} > k+1-J^{(k)}.
		\end{dcases}
	\end{align*}
	Since 
	\[
			\tau_{\mathsf{S}}^{(\theta, \infty)}, \mathfrak{T}_{\mathsf{S}}^{(\theta, \infty)}\leq \left\lceil\frac{\mathrm{diam}(S)}{w}\right \rceil + 1 =: c_0 \in \mathbb{N}
	\]
	by construction,
	for all $k \in \mathbb{N}$ follows
	\begin{align*}
		&\{\|\left( \boldsymbol{L}_{\mathsf{S}}^{(\theta, \infty)},\boldsymbol{R}_{\mathsf{S}}^{(\theta,\infty)}  \right) - \left( \boldsymbol{L}_{\mathsf{S}}^{(\theta,k)},\boldsymbol{R}_{\mathsf{S}}^{(\theta,k)}  \right)\| > \varepsilon\}\\
		&\quad\subseteq \{\left |   \boldsymbol{L}_{\mathsf{S}}^{(\theta,k)} -  \boldsymbol{L}_{\mathsf{S}}^{(\theta, \infty)}\right | > \frac{\varepsilon}{\sqrt{2}}\}
		\cup \{	\left |   \boldsymbol{R}_{\mathsf{S}}^{(\theta,\infty)} -  \boldsymbol{R}_{\mathsf{S}}^{(\theta, \infty)}\right |> \frac{\varepsilon}{\sqrt{2}}\}\\
		&\quad= \{\tau_{\mathsf{S}}^{(\theta,\infty)} > J^{(k)}\} \cup \{\mathfrak{T}_{\mathsf{S}}^{(\theta,k)} > k+1-J^{(k)}\}\\
		&\quad \subseteq \{c_0 > J^{(k)}\} \cup \{J^{(k)} > k+1-c_0\}.
	\end{align*}
	This yields
	\begin{align*}
		&\lim_{k \to \infty} \mathbb{P}\left(\|\left( \boldsymbol{L}_{\mathsf{S}}^{(\theta, \infty)},\boldsymbol{R}_{\mathsf{S}}^{(\theta,\infty)}  \right) - \left( \boldsymbol{L}_{\mathsf{S}}^{(\theta,k)},\boldsymbol{R}_{\mathsf{S}}^{(\theta,k)}  \right)\| > \varepsilon\right)\\
		&\quad \leq\lim_{k \to \infty} \mathbb{P}(c_0 > J^{(k)}) + \mathbb{P}(J^{(k)} > k+1-c_0)
		= \lim_{k \to \infty} 2\,\frac{c_0-1}{k} = 0.
	\end{align*}
\end{proof}

To state the next lemma set for $\alpha \in \mathbb{R}$
\begin{equation*}
	\Lambda_\alpha : \mathbb{R}  \to \mathbb{R}, \qquad
	\theta  \mapsto \alpha - \theta
\end{equation*}
and
\begin{equation*}
	\lambda_\alpha : \mathbb{R}^2  \to \mathbb{R}^2, \qquad
	(\ell, r)  \mapsto (\alpha - r, \alpha -\ell). 
\end{equation*}
\begin{lemma}
	Let $\mathsf{S} \in \mathcal{B}(\mathbb{R})$ with $\mathrm{diam}(\mathsf{S}) < \infty$ and $\theta, \alpha \in \mathbb{R}$.
	For any choice of hyperparameter $m \in \mathbb{N} \cup\{\infty\}$ which satisfy Assumption \ref{A: hyperparameters general}
	we have
	\begin{align*}
		\int_{\mathbb{R}^2} f(\ell, r) \mathbbm{1}_{(\ell, r)}(\alpha) \ \xi_{\mathsf{S}}^{(\theta)}\big(\d (\ell, r) \big)
		= \int_{\mathbb{R}^2} f(\ell, r) \mathbbm{1}_{(\ell, r)}(\theta) \ \xi_{\mathsf{S}}^{(\alpha)}\big(\d (\ell, r) \big)
	\end{align*}
	for all measurable functions $f: \mathbb{R}^2 \to [0,\infty)$, and 
	\[ 
		\xi_{\Lambda_\alpha(\mathsf{S})}^{(\theta)} = (\lambda_\alpha)_\sharp\xi_{\mathsf{S}}^{(\Lambda_\alpha(\theta))}.
	\]
\end{lemma}

\begin{proof}
	The additional superscript index indicates the choice of the hyperparameter $m$.
	For $m \in \mathbb{N}$ both statements are shown in \cite[Lemmas 11 and 12]{durmus2023geodesic}, such that it remains to treat the case $m= \infty$.
	By \cite[Theorem 13.11]{klenke2020probability} it suffices to show the first claim and 
	\[
		\int_{\mathbb{R}^2} f(\ell, r) \ \xi_{\Lambda_\alpha(\mathsf{S})}^{(\theta, \infty)}\big(\d (\ell, r) \big)
		= \int_{\mathbb{R}^2} f(\ell, r) \ (\lambda_\alpha)_\sharp\xi_{\mathsf{S}}^{(\Lambda_\alpha(\theta), \infty)}\big(\d (\ell, r) \big).
	\]
	for functions $f:\mathbb{R}^2\to [0,1]$ that are Lipschitz continuous with Lipschitz constant 1.
	Let $f$ be such a function.
	We treat the first claim.
	As $f$ is in particular continuous, the sets of discontinuities of $(\ell,r) \mapsto f(\ell,r) \mathbbm{1}_{(\ell, r)}(\alpha) $ and $(\ell,r) \mapsto f(\ell,r) \mathbbm{1}_{(\ell, r)}(\theta)$ are subsets of
	\begin{align*}
		(\{\alpha\} \times \mathbb{R} )\cup ( \mathbb{R} \times \{\alpha\}), \qquad  (\{\theta\} \times \mathbb{R} )\cup ( \mathbb{R} \times \{\theta\}),
	\end{align*}
	respectively.
	These are both nullsets with respect to $\xi_{\mathsf{S}}^{(\alpha, \infty)}$ and $\xi_{\mathsf{S}}^{(\theta, \infty)}$ by virtue of \cite[Lemma 19]{durmus2023geodesic}.
	Therefore combining Lemma \ref{L: Weak convergence stepping out} with the Portemanteau theorem (see \cite[Theorem 13.16]{klenke2020probability}) and the statement for finite hyperparameters $m$ yields
	\begin{align*}
		\int_{\mathbb{R}^2} f(\ell, r) \mathbbm{1}_{(\ell, r)}(\alpha) \ \xi_{\mathsf{S}}^{(\theta, \infty)}\big(\d (\ell, r) \big)
		&= \lim_{k \to \infty} \int_{\mathbb{R}^2} f(\ell, r) \mathbbm{1}_{(\ell, r)}(\alpha) \ \xi_{\mathsf{S}}^{(\theta, k)}\big(\d (\ell, r) \big)\\
		&= \lim_{k \to \infty} \int_{\mathbb{R}^2} f(\ell, r) \mathbbm{1}_{(\ell, r)}(\theta) \ \xi_{\mathsf{S}}^{(\alpha, k)}\big(\d (\ell, r) \big)\\
		&= \int_{\mathbb{R}^2} f(\ell, r) \mathbbm{1}_{(\ell, r)}(\theta) \ \xi_{\mathsf{S}}^{(\alpha, \infty)}\big(\d (\ell, r) \big).
	\end{align*}

	We turn to the second claim. As $\lambda_\alpha$ is an isometry, also $f \circ \lambda_\alpha: \mathbb{R}^2\to [0,1]$ is Lipschitz continuous with Lipschitz constant 1.
	Hence applying again Portemanteau theorem and the statement for finite hyperparameters $m$ yields
	\begin{align*}
		\int_{\mathbb{R}^2} f(\ell, r) \ \xi_{\Lambda_\alpha(\mathsf{S})}^{(\theta, \infty)}\big(\d (\ell, r) \big)
		&= \lim_{k \to \infty} 	\int_{\mathbb{R}^2} f(\ell, r) \ \xi_{\Lambda_\alpha(\mathsf{S})}^{(\theta, k)}\big(\d (\ell, r) \big)\\
		&= \lim_{k \to \infty} \int_{\mathbb{R}^2} f(\ell, r) \ (\lambda_\alpha)_\sharp\xi_{\mathsf{S}}^{(\Lambda_\alpha(\theta), k)}\big(\d (\ell, r) \big)\\
		&= \lim_{k \to \infty} \int_{\mathbb{R}^2} (f \circ \lambda_\alpha)(\ell, r) \ \xi_{\mathsf{S}}^{(\Lambda_\alpha(\theta), k)}\big(\d (\ell, r) \big)\\
		&= \int_{\mathbb{R}^2} (f \circ \lambda_\alpha)(\ell, r) \ \xi_{\mathsf{S}}^{(\Lambda_\alpha(\theta), \infty)}\big(\d (\ell, r) \big)\\
		&= \int_{\mathbb{R}^2} f(\ell, r) \ (\lambda_\alpha)_\sharp\xi_{\mathsf{S}}^{(\Lambda_\alpha(\theta), \infty)}\big(\d (\ell, r) \big).
	\end{align*}
\end{proof}

Finally, we establish the lemma needed in the proof of Corollary \ref{C: uniform ergodicity}.
\begin{lemma}\label{L: lower bound on stepping-out covering probability}
	Let $\mathsf{S} \in \mathcal{B}(\mathbb{R})$ and fix hyperparameters $m$ and $w$ that satisfy Assumption \ref{A: hyperparameters general}.
	Moreover, let $\theta \in \mathsf{S}$  and $C \in (\theta, \infty]$ and such that $b:= \sup \mathsf{S} \cap [\theta, C) < \infty$.
	If $(b-\theta)m^{-1}\mathbbm{1}_{(0, \infty)}(m)<w- \updelta \mathbbm{1}_{[2, \infty]}(m)$, then 
	\[
	\xi_{\mathsf{S}}^{(\theta)}(\mathsf{R})
	\geq 1 - \frac{b-\theta}{mw} \mathbbm{1}_{(0, \infty)}(m) - \frac{\updelta}{w}\mathbbm{1}_{[2, \infty]}(m)\qquad \in (0,1],
	\]
	where $\mathsf{R} := \left\{ (\ell, r) \in \mathbb{R}^2\ \mid\ [\theta, C) \cap \mathsf{S} \subseteq (\ell,r)\right\}$ and $\updelta := \mathrm{Leb}_{1}([\theta,b)\setminus \mathsf{S})$.
\end{lemma}

\begin{proof}
	Since $\boldsymbol{L}_{\mathsf{S}}^{(\theta)} < \theta$ be construction,
	we have $\{(\boldsymbol{L}_{\mathsf{S}}^{(\theta)}, \boldsymbol{R}_{\mathsf{S}}^{(\theta)}) \in \mathsf{R}\} = \{\boldsymbol{R}_{\mathsf{S}}^{(\theta)} \geq b\}$, such that $\xi_{\mathsf{S}}^{(\theta)}(\mathsf{R}) = 1 - \mathbb{P}(\boldsymbol{R}_{\mathsf{S}}^{(\theta)} < b)$.
	Define the not necessarily finite stopping time
	\[
	T:= \inf\{i \in \mathbb{N} \mid R_i^{(\theta)} \notin \mathsf{S}\}.
	\]
	With the convention $R_\infty^{(\theta)} \equiv \infty$, we get
	\[
	\left \{R^{(\theta)}_T < b \right \} = \bigsqcup_{k \in \mathbb{N}}\left \{ T=k, R^{(\theta)}_k < b\right \} \subseteq \bigcup_{k\in \mathbb{N}} \left \{R_k^{(\theta)} \in (\theta, b)\setminus \mathsf{S}\right \}.
	\]
	Therefore 
	\begin{align*}
		\mathbb{P}\left(R^{(\theta)}_T < b\right)
		&\leq \sum_{k\in \mathbb{N}} \frac{1}{w} \int_{(0,\,w)} \mathbbm{1}_{(\theta,\, b)\setminus \mathsf{S}}(\theta-u + kw)\ \mathrm{Leb}_{1}(\d u)\\
		&= \sum_{k\in \mathbb{N}} \frac{1}{w} \int_{(\theta+(k-1)w,\, \theta + kw)} \mathbbm{1}_{(\theta,\, b)\setminus \mathsf{S}}(v)\ \mathrm{Leb}_{1}(\d v)
		%		= \frac{\mathrm{Leb}_{1}([\theta,b)\setminus \mathsf{S})}{w}
		= \frac{\updelta}{w}.
	\end{align*}
	If $m= \infty$, then $\mathfrak{T}_{\mathsf{S}}^{(\theta)} = T$, such that in this case
	\[
	\xi_{\mathsf{S}}^{(\theta)}(\mathsf{R})
	= 1 - \mathbb{P}\left(R^{(\theta)}_T < b\right)
	\geq 1 - \frac{\updelta}{w},
	\]
	which is strictly positive by assumption.
	If $m <  \infty$,
	then $\{\boldsymbol{R}_{\mathsf{S}}^{(\theta)} < b\}\subseteq \{R^{(\theta)}_T < b, T < m+1-J\} \cup \{R^{(\theta)}_{m+1-J} < b\}$.
	We have
	\begin{align*}
		\mathbb{P}\left(R^{(\theta)}_{m+1-J} < b\right)
		&= \frac{1}{m} \sum_{k=1}^m \frac{1}{w} \int_{(0,\,w)} \mathbbm{1}_{(\theta,\,b)}\big(\theta - u + (m+1-k)w\big)\ \mathrm{Leb}_{1}(\d u)\\
		&= \frac{1}{mw} \sum_{k=1}^m  \int_{(\theta + (m-k)w,\,\theta + (m+1-k)w)} \mathbbm{1}_{(\theta,b)}(v)\ \mathrm{Leb}_{1}(\d v)\\
		&= \frac{1}{mw} \mathrm{Leb}_{1}\big((\theta,\,  b)\big)
		=\frac{b-\theta}{mw}.
	\end{align*}
	Moreover, if $m= 1$, then $\{T < m+1-J\} = \{T < 1\} = \emptyset$.
	Therefore in the case $m < \infty$ we get
	\begin{align*}
		\xi_{\mathsf{S}}^{(\theta)}(\mathsf{R}) 
		&\geq 1 - \mathbbm{1}_{[2, \infty]}(m)\mathbb{P}\left(R^{(\theta)}_T < b\right) - \mathbb{P}\left(R^{(\theta)}_{m+1-J} < b\right)\\
		&\geq 1 - \mathbbm{1}_{[2, \infty]}(m) \frac{\updelta}{w} - \frac{b-\theta}{mw},
	\end{align*}
	which is strictly positive by assumption.
\end{proof}

\section{Lower bound for the reeled shrinkage kernel}\label{S: shrinkage kernel}
In this section we provide the lower bound on the reeled shrinkage kernel which we use in the proof of the main statement. 
For the convenience of the reader we start by briefly restating the definition of the reeled shrinkage kernel.
A more gentle introduction can be found in \citep{hasenpflug2023reversibility, durmus2023geodesic}.

For $\alpha, \beta \in [0, 2\uppi)$, we use the shorthand notation
\[
I(\alpha, \beta) := \begin{dcases}
	[\alpha, \beta), & \alpha < \beta,\\
	[0, \beta) \cup [\alpha, 2\uppi), & \alpha > \beta,\\
	[0, 2\uppi), & \alpha = \beta,
\end{dcases}
\quad 
J(\alpha, \beta):= \begin{dcases}
	[\alpha, \beta), & \alpha < \beta,\\
	[0, \beta) \cup [\alpha, 2\uppi), & \alpha > \beta,\\
	\emptyset, & \alpha = \beta,
\end{dcases}
\]
and define $\Lambda_\theta := 	\{(\alpha, \alpha^{\min}, \alpha^{\max}) \in [0, 2\uppi)^3 \mid \alpha, \theta \in I(\alpha^{\min}, \alpha^{\max}), \alpha \neq \theta\}$ for $\theta \in [0, 2\uppi)$.
Indexed by $\theta \in [0, 2\uppi)$ let $(Z_n^{(\theta)})_{n \in \mathbb{N}} = (\Gamma_n^{(\theta)}, \Gamma_n^{(\theta), \min}, \Gamma_n^{(\theta), \max})_{n\in \mathbb{N}} $ be a Markov chain with transition kernel
\begin{align*}
	\Lambda_\theta \times \mathcal{B}([0, 2\uppi)^3)& \to [0,1]\\
	\big((\alpha, \alpha^{\min}, \alpha^{\max}), \mathsf{A}\big) &\mapsto \begin{multlined}[t]
		\mathbbm{1}_{I(\alpha, \alpha^{\max})}(\theta) \cdot \big(\mathrm{Unif}(I(\alpha, \alpha^{\max})) \otimes \delta_{(\alpha, \alpha^{\max})}\big) (\mathsf{A})\\
		+ \mathbbm{1}_{J(\alpha^{\min}, \alpha)}(\theta) \cdot \big(\mathrm{Unif}(I(\alpha^{\min}, \alpha))) \otimes \delta_{(\alpha^{\min}, \alpha)}\big) (\mathsf{A})
	\end{multlined}
\end{align*}
and initial distribution $\Gamma_1^{(\theta)} =\Gamma_1^{(\theta), \min}= \Gamma_1^{(\theta), \max}\sim \mathrm{Unif}([0, 2\uppi))$.
Now fix $\mathsf{S} \subseteq \mathbb{R}$ open, and $\ell,r \in \mathbb{R}$ with $\ell < r$ and $\mathsf{S} \cap (\ell,r) \neq \emptyset$.
Set 
\[
	h_{\ell,r} : [\ell,r) \to [0, 2\uppi), \qquad \theta \mapsto \frac{2\uppi}{r -\ell} \theta \mod 2 \uppi.
\]
For $\mathsf{A} \in \mathcal{B}(\mathbb{R})$ and $\theta \in (\ell,r)$ define
\begin{multline*}
	Q_{\mathsf{S}}^{\ell,r} (\theta,\mathsf{A}) := 
	\sum_{k=1}^{\infty} \mathbb{E}\bigg[ \prod_{i=1}^{k-1} \mathbbm{1}_{h_{\ell, r}(\mathsf{S} \cap (\ell,r))^\complement}(\Gamma_i^{(\theta)}) \mathbbm{1}_{I(\Gamma_k^{(\theta), \min}, \Gamma_k^{(\theta), \max})}\big(h_{\ell, r}(\theta)\big)\\
	\mathrm{Unif}\left (I(\Gamma_k^{(\theta), \min}, \Gamma_k^{(\theta), \max})\right )\Big(h_{\ell,r}\big(\mathsf{A} \cap \mathsf{S} \cap (\ell,r)\big)\Big)\bigg].
\end{multline*}
This is a transition kernel on $\mathsf{S} \cap (\ell, r)$, called the reeled shrinkage kernel.
We have the following lower bound.
\begin{lemma}\label{L: lower bound reeled shrinkage kernel}
	Let $\ell, r \in \mathbb{R}$ and $\mathsf{S} \subseteq \mathbb{R}$ open such that $\mathsf{S} \cap (\ell,r) \neq \emptyset$.
	For all $\theta \in \mathsf{S} \cap (\ell,r)$ holds
	\[
		Q^{\ell,r}_{\mathsf{S}}(\theta, \cdot) 
		\geq \frac{1}{\min\{r-\ell, \mathrm{diam}(\mathsf{S})\}}\ \mathrm{Leb}_{1}\big(\cdot \cap\, \mathsf{S} \cap (\ell,r)\big).
	\]
\end{lemma}

\begin{proof}
	Let $\mathsf{A} \in \mathcal{B}(\mathbb{R})$.
	Rewriting every summand, we get
	\begin{align*}
		Q^{\ell,r}_{\mathsf{S}}(\theta, \mathsf{A}) 
		& = \begin{multlined}[t]
			\sum_{k=1}^{\infty} \mathbb{E}\Bigg[ \prod_{i=1}^{k-1} \mathbbm{1}_{h_{\ell, r}(\mathsf{S} \cap (\ell,r))^\complement}(\Gamma_i^{(\theta)}) \mathbbm{1}_{I(\Gamma_k^{(\theta), \min}, \Gamma_k^{(\theta), \max})}\big(h_{\ell, r}(\theta)\big)\\
		\frac{\mathrm{Leb}_{1}\Big( h_{\ell,r}\big(\mathrm{conv}(\mathsf{S} \cap (\ell,r))\big)\Big)}{\mathrm{Leb}_{1}\big(I(\Gamma_k^{(\theta), \min}, \Gamma_k^{(\theta), \max})\big)}\\
		\frac{\mathrm{Leb}_{1}\left(h_{\ell,r}\big(\mathsf{A} \cap \mathsf{S} \cap (\ell,r)\big) \cap I(\Gamma_k^{(\theta), \min}, \Gamma_k^{(\theta), \max})\right)}{\mathrm{Leb}_{1}\Big( h_{\ell,r}\big(\mathrm{conv}(\mathsf{S} \cap (\ell,r))\big)\Big)} \Bigg]
		\end{multlined}\\
		& \geq \begin{multlined}[t]
			\sum_{k=1}^{\infty} \mathbb{E}\Bigg[ \prod_{i=1}^{k-1} \mathbbm{1}_{h_{\ell, r}(\mathsf{S} \cap (\ell,r))^\complement}(\Gamma_i^{(\theta)}) \mathbbm{1}_{I(\Gamma_k^{(\theta), \min}, \Gamma_k^{(\theta), \max})}\big(h_{\ell, r}(\theta)\big)\\
			\prod_{j=1}^{k-1} \mathbbm{1}_{h_{\ell,r}(\mathrm{conv} (\mathsf{S} \cap (\ell, r)))^\complement}(\Gamma_j^{(\theta)}) 
			\frac{\mathrm{Leb}_{1}\Big( h_{\ell,r}\big(\mathrm{conv} (\mathsf{S} \cap (\ell,r))\big)\Big)}{\mathrm{Leb}_{1}\big(I(\Gamma_k^{(\theta), \min}, \Gamma_k^{(\theta), \max})\big)}\\
			\frac{\mathrm{Leb}_{1}\left(h_{\ell,r}\big(\mathsf{A} \cap \mathsf{S} \cap (\ell,r)\big) \cap I(\Gamma_k^{(\theta), \min}, \Gamma_k^{(\theta), \max})\right)}{\mathrm{Leb}_{1}\Big( h_{\ell,r}\big(\mathrm{conv}(\mathsf{S} \cap (\ell,r))\big)\Big)} \Bigg]
		\end{multlined}
	\end{align*}
	Fix $k \in \mathbb{N}$.
	We show that 	$\Gamma_1^{(\theta)}, \ldots, \Gamma_{k-1}^{(\theta)} \notin h_{\ell,r}(\mathrm{conv}(\mathsf{S} \cap (\ell, r)))$ and $h_{\ell,r}(\theta) \in I(\Gamma_k^{(\theta), \min}, \Gamma_k^{(\theta), \max})$ implies that $h_{\ell,r}(\mathrm{conv}(\mathsf{S} \cap (\ell, r))) \subseteq I(\Gamma_k^{(\theta), \min}, \Gamma_k^{(\theta), \max})$ holds $\mathbb{P}$-almost surely.
	To this end, we equip $[0, 2\uppi)$ with the topology, where a set $\mathsf{B} \subseteq [0, 2\uppi)$ is called open if for all $x \in \mathsf{B}$ there exists $\upvarepsilon > 0$ such that $I(x-\upvarepsilon \mod 2\uppi, x+ \upvarepsilon \mod 2\uppi) \subseteq \mathsf{B}$.
	If 	$\Gamma_1^{(\theta)}, \ldots, \Gamma_{k-1}^{(\theta)}\notin h_{\ell,r}(\mathrm{conv} (\mathsf{S} \cap (\ell, r)))$,
	then $\Gamma_k^{(\theta), \min}, \Gamma_k^{(\theta), \max} \notin h_{\ell,r}(\mathrm{conv}(\mathsf{S} \cap (\ell, r)))$ holds $\mathbb{P}$-almost surely,
	because $\Gamma_k^{(\theta), \min}, \Gamma_k^{(\theta), \max} \in \{\Gamma_1^{(\theta)}, \ldots, \Gamma_{k-1}^{(\theta)}\}$ by construction $\mathbb{P}$-almost surely.
	As $I(\Gamma_k^{(\theta), \min}, \Gamma_k^{(\theta), \max})$ and $J( \Gamma_k^{(\theta), \max}, \Gamma_k^{(\theta), \min})$ form a partition of $[0, 2\uppi)$,
	this implies 
	\[
			h_{\ell,r}\Big(\mathrm{conv}\big(\mathsf{S} \cap (\ell,r)\big)\Big)
		\subseteq \mathrm{int}\big(I(\Gamma_k^{(\theta), \min}, \Gamma_k^{(\theta), \max})\big) \sqcup \mathrm{int}\big(J( \Gamma_k^{(\theta), \max}, \Gamma_k^{(\theta), \min})\big)
	\]
 	$\mathbb{P}$-almost surely.
	At the same time $h_{\ell,r}(\mathrm{conv}(\mathsf{S} \cap (\ell, r)))$ is connected, because $\mathrm{conv}(\mathsf{S} \cap (\ell,r))$ is connected and $h_{\ell,r}$ is continuous.
	Therefore $\mathbb{P}$-almost surely, we either have $	h_{\ell,r}\big(\mathrm{conv}(\mathsf{S} \cap (\ell,r))\big)
	\subseteq \mathrm{int}\big(I(\Gamma_k^{(\theta), \min}, \Gamma_k^{(\theta), \max})\big) $ or $	h_{\ell,r}\big(\mathrm{conv}(\mathsf{S} \cap (\ell,r))\big)
	\subseteq  \mathrm{int}\big(J( \Gamma_k^{(\theta), \max}, \Gamma_k^{(\theta), \min})\big)$, 
	because otherwise $h_{\ell,r}(\mathrm{conv} (\mathsf{S} \cap (\ell, r)))$ could be disjointly partitioned into two non-empty open sets.
	As $\theta \in \mathsf{S}$ by assumption, the condition $h_{\ell,r}(\theta) \in I(\Gamma_k^{(\theta), \min}, \Gamma_k^{(\theta), \max})$ only leaves 
	\begin{align*}
		h_{\ell,r}\Big(\mathrm{conv}\big(\mathsf{S} \cap (\ell,r)\big)\Big)
		&\subseteq \mathrm{int}\big(I(\Gamma_k^{(\theta), \min}, \Gamma_k^{(\theta), \max})\big) \\
		&\subseteq I(\Gamma_k^{(\theta), \min}, \Gamma_k^{(\theta), \max})\qquad \mathbb{P}\text{-almost surely}.
	\end{align*}
	Since in addition $h_{\ell,r}(\mathsf{S} \cap (\ell,r)) \subseteq h_{\ell,r}(\mathrm{conv}(\mathsf{S} \cap (\ell,r)))$ we obtain
	\begin{align*}
		Q^{\ell,r}_{\mathsf{S}}(\theta, \mathsf{A}) 
		&\geq \begin{multlined}[t]
				\frac{\mathrm{Leb}_{1}\left(h_{\ell,r}\big(\mathsf{A} \cap \mathsf{S} \cap (\ell,r)\big) \right)}{\mathrm{Leb}_{1}\Big( h_{\ell,r}\big(\mathrm{conv}\big(\mathsf{S} \cap (\ell,r)\big)\big)\Big)} 
				\sum_{k=1}^{\infty}\mathbb{E}\Bigg[ \prod_{i=1}^{k-1} \mathbbm{1}_{h_{\ell, r}(\mathsf{S} \cap (\ell,r))^\complement}(\Gamma_i^{(\theta)})\\ \mathbbm{1}_{I(\Gamma_k^{(\theta), \min}, \Gamma_k^{(\theta), \max})}\big(h_{\ell, r}(\theta)\big)
			\prod_{j=1}^{k-1} \mathbbm{1}_{h_{\ell,r}(\mathrm{conv} (\mathsf{S} \cap (\ell, r)))^\complement}(\Gamma_j^{(\theta)})\\ 
			\frac{\mathrm{Leb}_{1}\Big( h_{\ell,r}\big(\mathrm{conv} \big(\mathsf{S} \cap (\ell,r)\big)\big) \cap I(\Gamma_k^{(\theta), \min}, \Gamma_k^{(\theta), \max})\Big)}{\mathrm{Leb}_{1}\big(I(\Gamma_k^{(\theta), \min}, \Gamma_k^{(\theta), \max})\big)}\Bigg]
		\end{multlined}\\
		&= \begin{multlined}[t]
			\frac{\mathrm{Leb}_{1}\left(h_{\ell,r}\big(\mathsf{A} \cap \mathsf{S} \cap (\ell,r)\big) \right)}{\mathrm{Leb}_{1}\Big( h_{\ell,r}\big(\mathrm{conv}\big(\mathsf{S} \cap (\ell,r)\big)\big)\Big)} 
			\sum_{k=1}^{\infty}\mathbb{E}\Bigg[  \mathbbm{1}_{I(\Gamma_k^{(\theta), \min}, \Gamma_k^{(\theta), \max})}\big(h_{\ell, r}(\theta)\big)\\
			\prod_{j=1}^{k-1} \mathbbm{1}_{h_{\ell,r}(\mathrm{conv} (\mathsf{S} \cap (\ell, r)))^\complement}(\Gamma_j^{(\theta)}) \\
			\frac{\mathrm{Leb}_{1}\Big( h_{\ell,r}\big(\mathrm{conv} \big(\mathsf{S} \cap (\ell,r)\big)\big) \cap I(\Gamma_k^{(\theta), \min}, \Gamma_k^{(\theta), \max})\Big)}{\mathrm{Leb}_{1}\big(I(\Gamma_k^{(\theta), \min}, \Gamma_k^{(\theta), \max})\big)}\Bigg].
		\end{multlined}\\
	\end{align*}
	Note that the sum expression equals $ Q^{\ell,r}_{\mathrm{conv} (\mathsf{S})}(\theta, \mathrm{conv}(\mathsf{S})) = 1$.
	Hence, with a change of variables
	\begin{align*}
		Q^{\ell,r}_{\mathsf{S}}(\theta, \mathsf{A}) 
		&\geq \frac{\mathrm{Leb}_{1}\left(h_{\ell,r}\big(\mathsf{A} \cap \mathsf{S} \cap (\ell,r)\big) \right)}{\mathrm{Leb}_{1}\Big( h_{\ell,r}\big(\mathrm{conv}\big(\mathsf{S} \cap (\ell,r)\big)\big)\Big)}
		= \frac{\mathrm{Leb}_{1}\big(\mathsf{A} \cap \mathsf{S} \cap (\ell,r)\big)}{\mathrm{Leb}_{1}\big(\mathrm{conv}\big(\mathsf{S} \cap (\ell, r)\big)\big)}\\
		&\geq \frac{\mathrm{Leb}_{1}\big(\mathsf{A} \cap \mathsf{S} \cap (\ell,r)\big)}{\min\{r-\ell,\mathrm{Leb}_{1}(\mathrm{conv} (\mathsf{S}))\}}.
	\end{align*}
\end{proof}

\paragraph{Acknowledgements}
The author thanks Alain Durmus and Daniel Rudolf for fruitful discussions on this topic and valuable input for this paper.
Support by the DFG within project 432680300 – SFB 1456 subproject B02 is gratefully acknowledged.

 \bibliographystyle{amsalpha}
 \bibliography{sources.bib}  
\end{document}